\documentclass[12pt,reqno]{amsart}

\usepackage[cp1251]{inputenc}
\usepackage{amsthm,amssymb,amsmath,amsfonts}
\usepackage{graphicx}
\usepackage[matrix,arrow,curve]{xy}
\usepackage{longtable}
\newcommand{\Q}{\ensuremath{\mathbb{Q}}}

\newcommand{\Z}{\ensuremath{\mathbb{Z}}}

\newcommand{\F}{\ensuremath{\mathbb{F}}}
\newcommand{\Fr}{\ensuremath{\mathbf{F}}}

\newcommand{\ka}{\ensuremath{\Bbbk}}

\newcommand{\kka}{\ensuremath{\overline{\Bbbk}}}

\newcommand{\XX}{{\ensuremath{\overline{X}}}}

\newcommand{\ii}{\ensuremath{\mathrm{i}}}

\newcommand{\Pro}{\ensuremath{\mathbb{P}}}

\newcommand{\Aut}{\ensuremath{\operatorname{Aut}}}

\newcommand{\Gal}{\ensuremath{\operatorname{Gal}}}

\newcommand{\Pic}{\ensuremath{\operatorname{Pic}}}

\makeatletter
\@addtoreset{equation}{section}
\makeatother

\newtheorem{theorem}[equation]{Theorem}
\newtheorem{proposition}[equation]{Proposition}
\newtheorem{lemma}[equation]{Lemma}
\newtheorem{corollary}[equation]{Corollary}

\theoremstyle{definition}

\newtheorem{definition}[equation]{Definition}

\theoremstyle{remark}
\newtheorem{remark}[equation]{Remark}
\newtheorem*{notation}{Notation}

\title{Del Pezzo surfaces over finite fields}

\thanks{The research was carried out at the IITP RAS at the expense of the Russian Foundation for Sciences (project $N^{\underline{\circ}}$ 14-50-00150).}

\author{Andrey Trepalin}

\address{\emph{Andrey Trepalin}
\newline
\textnormal{Institute for Information Transmission Problems, 19 Bolshoy Karetnyi side-str., Moscow 127994, Russia}
\newline
\textnormal{\texttt{trepalin@mccme.ru}}}

\begin{document}

\begin{abstract}
Let $X$ be a del Pezzo surface of degree $2$ or greater over a finite field $\F_q$. The image $\Gamma$ of the Galois group $\Gal(\overline{\F}_q / \F_q)$ in the group $\Aut(\Pic(\XX))$ is a cyclic subgroup preserving the anticanonical class and the intersection form. The conjugacy class of $\Gamma$ in the subgroup of $\Aut(\Pic(\XX))$ preserving the anticanonical class and the intersection form is a natural invariant of $X$. We say that the conjugacy class of $\Gamma$ in $\Aut(\Pic(\XX))$ is the \textit{type} of a del Pezzo surface. In this paper we study which types of del Pezzo surfaces of degree $2$ or greater can be realized for given $q$. We collect known results about this problem and fill the gaps.
\end{abstract}

\maketitle
\section{Introduction}

A \textit{del Pezzo surface} is a smooth projective surface such that its anticanonical class is ample. Let $X$ be a del Pezzo surface over finite field $\F_q$, and $\XX = X \otimes \overline{\F}_q$. The Galois group~$\Gal\left(\overline{\F}_q / \F_q\right)$ acts on the lattice $\Pic(\XX)$ and preserves the anticanonical class~$-K_X$ and the intersection form. The image $\Gamma$ of $\Gal\left(\overline{\F}_q / \F_q\right)$ is a cyclic subgroup in~$\Aut\left(\Pic(\XX)\right)$ preserving the intersection form. Obviously, the set of elements in $\Aut\left(\Pic(\XX)\right)$ preserving the intersection form is a group. We denote this group by $W$.

The group $\Gamma$ is a cyclic subgroup of $W$. Therefore the conjugacy class of $\Gamma$ in $W$ is a natural invariant of a del Pezzo surface. We say that the conjugacy class of $\Gamma$ in $W$ is the \textit{type} of a del Pezzo surface. It is well-known (see, for example, \cite[Theorem IV.1.1]{Man74}) that the type of a del Pezzo surface defines the number $N_1$ of $\F_q$-points on $X$. One has
\begin{equation}
\label{Fpoints}
N_1 = q^2 + aq + 1,
\end{equation}
\noindent where $a$ is the trace of a generator $g$ of $\Gamma$ considered as an element on $\mathrm{GL}(\Pic(\XX))$.

More generally, let $N_k$ be the number of $\F_{q^k}$ points of $X_k = X \otimes \F_{q^k}$. The zeta function of $X$ is the formal power series
$$
Z_X(t)=\exp\left(\sum_{k=1}^\infty \frac{N_kt^k}{k}\right).
$$
For a geometrically rational surface $X$ one has (see~\cite[Corollary 2 from Theorem~IV.5.1]{Man74})
$$
Z_X(t)=\frac{1}{(1-t)(1-q^2t)\det(1-qtg|\Pic(\XX)\otimes\Q)}.
$$
\noindent Therefore the zeta function of a del Pezzo surface $X$ is totally defined by the type of $X$. Thus the type of $X$ defines the numbers $N_k$, since these numbers are uniquely determined by a given zeta function.

The study of algebraic varieties over finite fields has numerous applications in coding theory (see, for example, \cite{VNT18}).

The natural question is which types of del Pezzo surfaces can be realized for given $q$. In \cite[Theorem~1.7]{BFL16} it is shown that a del Pezzo surface of any type exists for any sufficiently big $q$. But there is no answer to the question for small $q$. The aim of this paper is to construct each type of del Pezzo surfaces of degree $2$ or greater over $\F_q$ if it is possible, and show that such surfaces do not exist for other values of $q$. 

The \textit{degree} of a del Pezzo surface $X$ is the number $d = K_X^2$. One has $1 \leqslant d \leqslant 9$. For each $d$ we give the isomorphism class of $W$ in the following table, where $W(E_6)$, $W(E_7)$ and $W(E_8)$ are the Weyl groups of the corresponding root systems $E_6$, $E_7$ and $E_8$.

\begin{center}
\begin{tabular}{|c|c|c|c|c|c|c|c|c|c|}
\hline
Degree & $9$ & $8$ & $7$ & $6$ & $5$ & $4$ & $3$ & $2$ & $1$ \\
\hline
Group & $\langle \mathrm{id} \rangle$ & $\langle \mathrm{id} \rangle$ or $\Z / 2\Z$ & $\Z / 2\Z$ & $D_6$ & $S_5$ & $\left(\Z / 2\Z\right)^4 \rtimes S_5$ & $W(E_6)$ & $W(E_7)$ & $W(E_8)$ \\
\hline
\end{tabular}
\end{center}

One can easily find conjugacy classes of elements in the dihedral group $D_6$ of order $12$ and the symmertic group $S_5$. For the group $\left(\Z / 2\Z\right)^4 \rtimes S_5$ it is more difficult, and we give a proof of the classification in Proposition \ref{nomistakeDP4}, since in some sources there are missed cases in the classification of finite subgroups of $\left(\Z / 2\Z\right)^4 \rtimes S_5$. The classification of conjugacy classes of the elements in the groups $W(E_6)$, $W(E_7)$ and $W(E_8)$ is obtained in \cite{Fr51} and~\cite{Fr67}. 

A surface $S$ is called \textit{minimal} if any birational morphism $S \rightarrow S'$ is an isomorphism. The type of a del Pezzo surface $X$ allows to determine whether the surface is minimal or not. If $X$ is not a minimal surface then it is isomorphic to the blowup of a surface~$Y$ at some points of certain degrees. Such points of the blowup should satisfy some conditions that are called \textit{being in general position} (see Theorem \ref{GenPos}, Proposition \ref{DP8GenPos} and Corollary~\ref{dPblowup}). An important case of nonminimal del Pezzo surfaces is the blowup of the projective plane at some points of certain degrees in general position. Surfaces of these types have many interesting properties and are studied in many papers. For example, the blowups of $\Pro^2_{\F_q}$ at six $\F_q$-points or four $\F_q$-points and a point of degree two are considered in \cite{SD10}, and the blowup of $\Pro^2_{\F_q}$ at seven $\F_q$-points is considered in \cite[Chapter 4, Sections 2 and~3]{Kap13} and \cite{KR16} for odd $q$. Many types (including most complicated) of nonminimal del Pezzo surfaces are considered in \cite{BFL16}, but the aim of \cite{BFL16} is to construct del Pezzo surfaces with a given number of $\F_q$-points. Some types of nonminimal del Pezzo surfaces are not considered in \cite{BFL16} because del Pezzo surfaces of different types can have the same number of $\F_q$-points.

The types of minimal del Pezzo surfaces of degrees $4$, $3$ and $2$ are considered in the papers of S.\,Rybakov and the author. Minimal del Pezzo surfaces of degree $4$ are constructed in \cite[Theorem~3.2]{Ry05}. These surfaces either admit a structure of a conic bundle with four degenerate fibres, or are isomorphic to the contracting of a $(-1)$-curve on a del Pezzo surface of degree $3$ admitting a structure of a conic bundle with five degenerate fibres. Any del Pezzo surface of degree $3$ is a cubic surface in $\Pro^3_{\F_q}$. There are five types of minimal cubic surfaces. One type is constructed in \cite{SD10} and the other four types are constructed in \cite[Theorem~1.2]{RT17}, but there are some restrictions on $q$ for three types. Minimal del Pezzo surfaces of degree $2$ are constructed in \cite{Tr17}. For four types there are some restrictions on $q$ (three of them are inherited from the restrictions on $q$ for the case of cubic surfaces).

We give the classification of the types of del Pezzo surfaces of degree $5$ or greater, $4$, $3$ and $2$ in Tables \ref{table1}, \ref{table4}, \ref{table3} and \ref{table2} respectively.

The main results of this paper are the following three theorems.

\begin{theorem}
\label{DP2}

In the notation of Table \ref{table2} the following holds.

\begin{enumerate}
\item Del Pezzo surfaces of degree $2$ of types $1$ and $49$ do not exist over $\F_2$, $\F_3$, $\F_4$, $\F_5$, $\F_7$, $\F_8$, and exist over other finite fields.

\item Del Pezzo surfaces of degree $2$ of types $5$ and $10$ do not exist over $\F_2$, $\F_3$, $\F_4$, $\F_5$, and exist over other finite fields.

\item Del Pezzo surfaces of degree $2$ of types $2$, $3$, $18$ and $31$ do not exist over $\F_2$, $\F_3$, $\F_4$, and exist over other finite fields.

\item Del Pezzo surfaces of degree $2$ of types $4$, $6$--$9$, $12$--$14$, $17$, $21$, $22$, $25$, $32$, $33$, $38$, $40$, $50$, $53$, $55$, $60$ do not exist over $\F_2$, and exist over other finite fields.

\item Del Pezzo surfaces of degree $2$ of types $11$, $15$, $16$, $19$, $20$, $23$, $24$, $26$, $27$, $29$, $30$, $34$, $36$, $37$, $39$, $41$--$46$, $48$, $51$, $52$, $54$, $57$--$59$ exist over all finite fields.

\item Del Pezzo surfaces of degree $2$ of types $28$ and $35$ do not exist over $\F_2$, and exist over any~$\F_q$ where $q \geqslant 4$.

\item Del Pezzo surfaces of degree $2$ of types $47$ and $56$ exist over any $\F_q$ where $q = 6k + 1$.

\end{enumerate}

\end{theorem}

\begin{theorem}
\label{DP3}

In the notation of Table \ref{table3} the following holds.

\begin{enumerate}
\item A cubic surface of type $(c_1)$ does not exist over $\F_2$, $\F_3$, $\F_5$, and exists over other finite fields.

\item Cubic surfaces of types $(c_2)$, $(c_3)$ do not exist over $\F_2$, $\F_3$, and exist over other finite fields.

\item Cubic surfaces of types $(c_4)$, $(c_5)$, $(c_9)$, $(c_{10})$ do not exist over $\F_2$, and exist over other finite fields.

\item Cubic surfaces of types $(c_6)$--$(c_8)$, $(c_{11})$--$(c_{13})$ and $(c_{15})$--$(c_{25})$ exist over all finite fields.

\item A cubic surface of type $(c_{14})$ exists over any $\F_q$ where $q = 6k + 1$.

\end{enumerate}

\end{theorem}

\begin{theorem}
\label{DP4}

In the notation of Table \ref{table4} the following holds.

\begin{enumerate}
\item Del Pezzo surfaces of degree $4$ of types $\mathrm{id}$, $\iota_{ab}$, $\iota_{abcd}$ do not exist over $\F_2$, $\F_3$, and exist over other finite fields.

\item Del Pezzo surfaces of degree $4$ of types $(ab)(cd)$, $(ab)(cd)\iota_{ac}$, $(ab)(cd)\iota_{ae}$ do not exist over $\F_2$, and exist over other finite fields.

\item The other types of del Pezzo surfaces of degree $4$ exist over all finite fields.

\end{enumerate}

\end{theorem}

The plan of this paper is as follows.

In Section $2$ we remind some notions about del Pezzo surfaces and the corresponding Weyl groups.

In Section $3$ we give a short proof of a well-known fact, that del Pezzo surfaces of degree~$5$ or greater of any type exist over any $\F_q$, and give the classification of the types of del Pezzo surfaces of degree $5$ or greater.

In Sections $4$--$6$ we prove Theorems \ref{DP2}, \ref{DP3} and \ref{DP4}. By Corollary \ref{dPblowup} if a del Pezzo surface $Y$ of degree $d \geqslant 3$ contains an $\F_q$-point not lying on the lines, then the blowup of $Y$ at this point a del Pezzo surface $X$ of degree $d -1$. Moreover, the type of $X$ is defined by the type of $Y$. Therefore surfaces of the same type as $Y$ exist over any finite field, such that there exists a surface $X$ of the corresponding type. Thus our strategy is the following. We start from del Pezzo surfaces of degree $2$, then we pass to del Pezzo surfaces of degree~$3$ and consider only those types, for which del Pezzo surfaces of the corresponding types do not exist over all finite fields. After this we consider the types of del Pezzo surfaces of degree $4$ such that their existence over all finite fields does not immediately follows from the existence of del Pezzo surface of degree $3$ of the corresponding types. 

In Section $4$ we consider del Pezzo surfaces of degree $2$ and prove Theorem \ref{DP2}. In Table~\ref{table2} we give the classification of conjugacy classes in $W(E_7)$. These $60$ conjugacy classes are divided into $30$ pairs, such that the Geiser twist (see Definition \ref{GeiserTwistDef}) of a surface of one type in a pair has the other type in the pair. The surfaces of types, that are paired by Geiser twist, exist over the same finite fields. Many pairs of types are considered in \cite{BFL16} and \cite{Tr17}. For each of the remaining pairs of types one type in the pair can be constructed as either the blowup of $\Pro^2_{\F_q}$ at some points of certain degrees, or the blowup of a minimal del Pezzo surface of degree $4$ at two $\F_q$-points, or the blowup of a cubic surface at an $\F_q$-point. We give the constructions of del Pezzo surfaces of degree $2$ of these types over $\F_q$ if it is possible, and show that such types of surfaces do not exist for other values of $q$.

In Section $5$ we consider del Pezzo surfaces of degree $3$ (that are cubic surfaces) and prove Theorem \ref{DP3}. In Table~\ref{table3} we give the classification of conjugacy classes in $W(E_6)$. There are $25$ types of cubic surfaces. Existence of cubic surfaces of eight types over all finite fields immediately follows from the existence of the corresponding (see the last column of Table \ref{table3}) del Pezzo surfaces of degree $2$ over all finite fields. Many types are considered in \cite{Ry05}, \cite{SD10}, \cite{RT17}, \cite{Tr17} and \cite{BFL16}. Some types of cubic surfaces are considered in the proofs of some lemmas in Section $4$. The remaining three types can be constructed as the blowup of $\Pro^2_{\F_q}$ at some points of certain degrees. We give the constructions of cubic surfaces of these types over $\F_q$ if it is possible, and show that such types of surfaces do not exist for other values of $q$.

In Section $6$ we consider del Pezzo surfaces of degree $4$ and prove Theorem \ref{DP4}. In Table~\ref{table4} we give the classification of conjugacy classes in $W(D_5) \cong \left(\Z / 2\Z\right)^4 \rtimes S_5$. There are $18$ types of del Pezzo surfaces of degree $4$. For almost all types del Pezzo surfaces of degree $4$ exist over all finite fields, since cubic surfaces of the corresponding (see the last column of Table \ref{table4}) types exist over all finite fields. Three of the remaining types are considered in \cite{Ry05} and \cite{BFL16}. The remaining three types can be constructed as the blowup of $\Pro^2_{\F_q}$ or a smooth quadric in $\Pro^3_{\F_q}$ at some points of certain degrees. We give the constructions of del Pezzo surfaces of degree $4$ of these types over $\F_q$ if it is possible, and show that such types of surfaces do not exist for other values of $q$.

In Section $7$ we discuss problems arising for del Pezzo surfaces of degree $1$, and show that for almost all methods of constructing of del Pezzo surfaces of degree $2$ or greater there appear some additional difficulties.

The author is a Young Russian Mathematics award winner and would like to thank its sponsors and jury. Also the author is grateful to S.\,Rybakov for valuable comments. 

\begin{notation}

Throughout this paper all surfaces are smooth, projective and defined over a finite field $\F_q$ of order $q$, where $q$ is a power of prime. For a surface $X$ we denote~$X \otimes \overline{\F}_q$ by $\XX$. For a del Pezzo surface $X$ we denote the Picard group by $\Pic(X)$, the image of the group~$\Gal\left( \overline{\F}_q / \F_q\right)$ in the corresponding Weyl group $W(R)$ acting on the Picard group~$\Pic(\XX)$ is denoted by $\Gamma$, for the $\Gamma$-invariant Picard group $\Pic(\XX)^{\Gamma}$ one \mbox{has $\Pic(\XX)^{\Gamma} = \Pic(X)$}. The number $\rho(X) = \operatorname{rk} \Pic(X)$ is the Picard number of $X$, for the $\Gamma$-invariant Picard number $\rho(\XX)^{\Gamma}$ one has $\rho(\XX)^{\Gamma} = \rho(X)$. The subspace of classes $C$ in $\Pic(\XX) \otimes \mathbb{Q}$ such that $C \cdot K_X = 0$ is denoted by $K_X^{\perp}$.

For any variety $X$ we denote by $\Fr$ the Frobenius automorphism of $X$.

We denote by $\xi_d$ a primitive root of unity of order $d$, and $\omega = \xi_3$, $\ii = \xi_4$. 

\end{notation}

\section{Del Pezzo surfaces and Weyl groups}

In this section we remind some basic notions about del Pezzo surfaces and the corresponding Weyl groups.

\begin{definition}
\label{DPdef}
A smooth projective surface $X$ such that the anticanonical class $-K_X$ is ample is called a {\it del Pezzo surface}.

The number $d = K_X^2$ is called the {\it degree} of a del Pezzo surface.
\end{definition}

It is well known that a del Pezzo surface $\XX$ over an algebraically closed field $\kka$ is isomorphic to $\Pro^2_{\kka}$, $\Pro^1_{\kka} \times \Pro^1_{\kka}$ or the blowup of $\Pro^2_{\kka}$ at up to $8$ points in \textit{general position}. More precisely, the following theorem holds.

\begin{theorem}[cf. {\cite[Theorem IV.2.5]{Man74}}]
\label{GenPos}
Let $1 \leqslant d \leqslant 9$, and $p_1$, $\ldots$, $p_{9-d}$ be $9-d$ points on the projective plane $\Pro^2_{\kka}$ such that
\begin{itemize}
\item no three lie on a line;
\item no six lie on a conic;
\item for $d = 1$ the points are not on a singular cubic curve with singularity at one of these points.
\end{itemize}
Then the blowup of $\Pro^2_{\kka}$ at $p_1$, $\ldots$, $p_{9-d}$ is a del Pezzo surface of degree $d$.

Moreover, any del Pezzo surface $\XX$ of degree $1 \leqslant d \leqslant 7$ over algebraically closed field $\kka$ is the blowup of such set of points.
\end{theorem}

\begin{definition}
\label{GenPosdef}
We say that a collection of geometric points on $\Pro^2_{\ka}$ is \textit{in general position} if it satisfies the conditions of Theorem \ref{GenPos}.

Moreover, for a number of points of certain degree on $\Pro^2_{\F_q}$ we say that these points are \textit{in general position} if the corresponding geometric points are in general position.
\end{definition}

\begin{definition}
\label{lines}
A curve $E$ of genus $0$ on a surface such that $E^2 = -1$ is called \mbox{a \textit{$(-1)$-curve}}.

For a surface $X$ one has $(K_X + E)\cdot E = -2$, thus if $X$ is a del Pezzo surface \mbox{then $-K_X \cdot E = 1$}, and the anticanonical morphism $\varphi_{|-K_X|}$ maps $E$ to a line in a projective space of dimension $d = K_X^2$ for $d \geqslant 2$.

Therefore if there is no confusion $(-1)$-curves on del Pezzo surfaces are called \textit{lines}.
\end{definition}

The following theorem describes the set of $(-1)$-curves on del Pezzo surfaces.

\begin{theorem}[cf. {\cite[Theorem IV.4.3]{Man74}}]
\label{DPlines}
Let $\XX$ be a del Pezzo surface of degree $1 \leqslant d \leqslant 7$ and $f: \XX \rightarrow \Pro^2_{\kka}$ be the blowup of $\Pro^2_{\kka}$ at points $p_1$, \ldots, $p_{9-d}$. Then the set of $(-1)$-curves on $\XX$ consists of the following curves:
\begin{itemize}
\item the preimages of $p_i$;
\item the proper transforms of lines passing through two points from the set $\{p_i\}$;
\item the proper transforms of conics passing through five points from the set $\{p_i\}$;
\item the proper transforms of cubics passing through seven points from the set $\{p_i\}$ and passing through one of these points with multiplicity $2$;
\item the proper transforms of quartics passing through eight points from the set $\{p_i\}$ and passing through three of these points with multiplicity $2$;
\item the proper transforms of quintics passing through eight points from the set $\{p_i\}$ and passing through six of these points with multiplicity $2$;
\item the proper transforms of sextics passing through eight points from the set $\{p_i\}$ with multiplicity at least $2$ and passing through one of these points with multiplicity $3$.
\end{itemize}

\end{theorem}

Applying Theorem \ref{DPlines} one can compute the numbers of $(-1)$-curves on del Pezzo surfaces, that are presented in the following table.

\begin{center}
\begin{tabular}{|c|c|c|c|c|c|c|c|c|c|}
\hline
Degree & $9$ & $8$ & $7$ & $6$ & $5$ & $4$ & $3$ & $2$ & $1$ \\
\hline
Number of lines & $0$ & $1$ or $0$ & $3$ & $6$ & $10$ & $16$ & $27$ & $56$ & $240$ \\
\hline
\end{tabular}
\end{center}

The following useful corollary from Theorem \ref{GenPos} is well-known. For the proof see, for example, \cite[Corollary 3.9]{Tr17}.

\begin{corollary}
\label{dPblowup}
Let $X$ be a del Pezzo surface of degree $3 \leqslant d \leqslant 9$, and let $p$ be \mbox{an $\F_q$-point} that does not lie on the lines. Then the blowup of $X$ at $p$ is a del Pezzo surface of degree~$d-1$.
\end{corollary}

To apply Corollary \ref{dPblowup} one should compare the number of $\F_q$-points on $X$ and the number of $\F_q$-points on $X$ lying on the lines. The number of $\F_q$-points on $X$ of certain type is given by equation \eqref{Fpoints}. The following remark allows us to compute the number of~$\F_q$-points on $X$ lying on the lines.

\begin{remark}
\label{Linespoints}
Let $X$ be a del Pezzo surface of degree $3 \leqslant d \leqslant 7$. Then any line defined over $\F_q$ contains $q + 1$ points defined over $\F_q$.

If $d \geqslant 4$ then there are no points of intersection of three or more lines. Therefore the number of $\F_q$-points on $X$ lying on the lines is equal to $A(q + 1) - B + C$, where $A$ is the number of lines defined over $\F_q$, $B$ is the number of pairs of meeting each other lines defined over $\F_q$, and $C$ is the number of pairs of meeting each other conjugate lines defined over $\F_{q^2}$.

If $d = 3$ then then there are no points of intersection of four or more lines. Three lines $H_1$, $H_2$ and $H_3$ meet each other, if and only if the divisor $H_1 + H_2 + H_3$ is linearly equivalent to $-K_X$. For such triple of lines there are two possibilities: either there are three distinct points $H_1 \cap H_2$, $H_1 \cap H_3$, $H_2 \cap H_3$, or the three lines have a common point, that is called an \textit{Eckardt point}. Therefore if all three lines are defined over $\F_q$, then the union of these lines contains $3q$ or $3q + 1$ points defined over $\F_q$. If one of these lines is defined over $\F_q$, and the others are conjugate and defined over $\F_{q^2}$, then the union of these lines contains $q + 2$ or $q + 1$ points defined over $\F_q$. If these three lines are conjugate and defined over $\F_{q^3}$, then the union of these lines contains $0$ or $1$ point defined over $\F_q$. In all other cases there are no $\F_q$-points on the union of these three lines.

Therefore for $d = 3$ the number of $\F_q$-points on $X$ lying on the lines is equal \mbox{to $A(q + 1) - B + C + D - E + F$}, where $A$ is the number of lines defined over $\F_q$, $B$ is the number of pairs of meeting each other lines defined over $\F_q$, $C$ is the number of pairs of meeting each other conjugate lines defined over $\F_{q^2}$, $D$ is the number of Eckardt points lying on a triple of lines defined over $\F_q$, $E$ is the number of Eckardt points lying on a triple of lines defined over $\F_{q^2}$ such that two of these lines are conjugate, $F$ is the number of Eckardt points lying on a triple of conjugate lines defined over $\F_{q^3}$. Note that the type of a cubic surface $X$ defines the numbers $A$, $B$ and $C$. Therefore the number of $\F_q$-points on $X$ lying on the lines is at least $A(q + 1) - B + C - E$ and at most $A(q + 1) - B + C + D + F$.

\end{remark} 

To count the mentioned in Remark \ref{Linespoints} numbers of Eckardt points we need the following lemma.

\begin{lemma}[{\cite[Lemma 9.4]{DD17}}]
\label{Eckardtlines}
A line on a cubic surface contains $0$, $1$ or $2$ Eckardt points if $q$ is odd, and contains $0$, $1$ or $5$ Eckardt points if $q$ is even.
\end{lemma}

If $\XX$ is not isomorphic to $\Pro^1_{\kka} \times \Pro^1_{\kka}$, then the Picard group of $\XX$ is generated by the proper transform $L$ of the class of a line on $\Pro^2_{\kka}$, and the classes $E_1$, \ldots, $E_{9-d}$ of the exceptional divisors. For $d \leqslant 6$ the subspace $K_X^{\perp} \subset \Pic\left(\XX\right) \otimes \mathbb{Q}$ is generated by
$$
L - E_1 - E_2 - E_3, \qquad E_1 - E_2, \qquad E_2 - E_3, \qquad \ldots, \qquad E_{8-d} - E_{9 - d}.
$$
This set of generators are simple roots for the root system of a certain type given in the following table (for details see e. g. \cite[Theorem IV.3.5]{Man74}). 

\begin{center}
\begin{tabular}{|c|c|c|c|c|c|c|}
\hline
Degree & $6$ & $5$ & $4$ & $3$ & $2$ & $1$ \\
\hline
Root system & $A_2 \times A_1$ & $A_4$ & $D_5$ & $E_6$ & $E_7$ & $E_8$ \\
\hline
\end{tabular}
\end{center}

To simplify the notation we denote by $E_{9-d}$ the root system corresponding to a del Pezzo surface of degree $d$.

Any group, acting on the Picard lattice $\Pic\left(\XX\right)$ and preserving the intersection form, is a subgroup of the Weyl group $W(E_{9 - d})$. In particular, if $X$ is a del Pezzo surface over a field $\ka$, then the group $\Gal\left(\kka / \ka \right)$ acts on $\Pic(\XX)$, and its image $\Gamma$ in $\Aut\left(\Pic\left(\XX\right)\right)$ is a subgroup of $W(E_{9 - d})$. Moreover, if $\ka$ is a finite field then $\Gamma$ is a cyclic subgroup of~$W(E_{9 - d})$. One can easily classify finite subgroups of $W(A_2 \times A_1) \cong D_6$ that is a dihedral group of order $12$, $W(A_4) \cong S_5$ and $W(D_5) \cong \left(\Z / 2\Z\right)^4 \rtimes S_5$. The classification of finite subgroups in $W(E_6)$, $W(E_7)$, $W(E_8)$ is obtained in \cite{Car72}.

One of the main properties of a conjugacy class in $W(E_{9 - d})$ is a Carter graph, that was introduced in \cite{Car72}. This graph describes eigenvalues of the action of $\Gamma$ \mbox{on $K_X^{\perp} \subset \Pic\left(\XX\right) \otimes \mathbb{Q}$} (see \cite[Table 3]{Car72}), and, in particular, defines the order of $\Gamma$. Moreover, we have the following useful well-known lemma.

\begin{lemma}
\label{blowup}
Let $Y$ be a del Pezzo surface of type whose Carter graph is $R$, and $X$ be the blowup of $Y$ at a point $P$ of degree $k$. If $X$ is a del Pezzo surface then its type has the Carter graph $R \times A_{k-1}$. In particular, if $P$ is an $\F_q$-point then the types of $X$ and $Y$ have the same Carter graph.
\end{lemma}

\section{Del Pezzo surfaces of degree $5$ and greater}

In this section we give a proof of the well-known fact that del Pezzo surfaces of degree~$5$ or greater of any types exist over all finite fields. For the sake of completeness we give the classification of the types of del Pezzo surface of degree $5$ or greater.

\begin{proposition}
\label{DP5}
Del Pezzo surfaces of degree $5$ or greater of any types exist over all finite fields.
\end{proposition}

\begin{proof}

For del Pezzo surfaces of degree $5$ the group $\Gamma$ is a cyclic subgroup of the Weyl group $W(A_4) \cong S_5$. The conjugacy class of a cyclic subgroup $\Gamma$ in $S_5$ is defined by the cyclic type of $\sigma \in S_5$.

For a given $\sigma$ consider a smooth conic $Q$ on $\Pro^2_{\F_q}$ and five geometric points $p_1$, \ldots, $p_5$ on~$Q$, such that $\Fr$ acts on these points as $p_i \mapsto p_{\sigma(i)}$. Consider the blowup $f: X_4 \rightarrow \Pro^2_{\F_q}$ of the points $p_1$, \ldots, $p_5$. The surface $X_4$ is a del Pezzo surface of degree $4$ by Theorem~\ref{GenPos}, since if three points $p_i$ lie on a line $l$ then $l \cdot Q \geqslant 3$ that is impossible. One has $f_*^{-1}(Q)^2 = -1$. Therefore there exists the contraction $\pi: X_4 \rightarrow X_5$ of $f_*^{-1}(Q)$, where $X_5$ is a del Pezzo surface of degree $5$ of type corresponding to $\sigma$.

This method works in all cases except the trivial group and the group $\Z / 2\Z \cong \langle (12)(34) \rangle$ over $\F_2$, since there are only three $\F_2$-points and one point of degree $2$ on a conic. But in these cases one can consider a blowup $X_5 \rightarrow \Pro^2_{\F_2}$ at four $\F_2$-points or two points of degree~$2$ in general position respectively.

On a del Pezzo surface of degree $6$--$9$ one can blow up a number of $\F_q$-points and get a del Pezzo surface of degree $5$. One can consider a del Pezzo surface of degree $5$ of the corresponding type and contract a number of $(-1)$-curves defined over $\F_q$.

\end{proof}

In Table \ref{table1} for del Pezzo surfaces $X$ of degree $5$ and greater we give the classification of conjugacy classes in the groups acting on $\Pic(\XX)$ and preserving $K_X$ and the intersection form. The first column is the isomorphism class of the surface $\XX = X \otimes \overline{\F}_q$. In this column we denote by $X_d$ del Pezzo surfaces such that $\XX_d$ can be obtained as the blowup of $\Pro^2_{\overline{\F}_q}$ at $9 - d$ points. The second column is a conjugacy class of the generator of $\Gamma$. We denote the generator of $\Z / 2\Z$ by $g$, and denote the generators of $D_6$ by $r$ and $s$, with the relations $r^6 = s^2 = srsr = 1$. The third column is a Carter graph corresponding to the conjugacy class (see \cite{Car72}). The fourth column is the order of an element. The fifth column is the collection of eigenvalues of the action of an element on $K_X^{\perp} \subset \Pic(\XX) \otimes \mathbb{Q}$. The sixth column is the invariant Picard number $\rho(\XX)^{\Gamma}$.

\begin{table}

\begin{center}
\begin{tabular}{|c|c|c|c|c|c|}
\hline
Surface & Class & Graph & Order & Eigenvalues & $\rho(\XX)^{\Gamma}$ \\
\hline
$\Pro^2$ & $\mathrm{id}$ & $\varnothing$ & $1$ & ~ & $1$ \\
\hline
\hline
$X_8$ & $\mathrm{id}$ & $\varnothing$ & $1$ & $1$ & $2$ \\
\hline
\hline
$\Pro^1 \times \Pro^1$ & $\mathrm{id}$ & $\varnothing$ & $1$ & $1$ & $2$ \\
\hline
$\Pro^1 \times \Pro^1$ & $g$ & $A_1$ & $2$ & $-1$ & $1$ \\
\hline
\hline
$X_7$ & $\mathrm{id}$ & $\varnothing$ & $1$ & $1$, $1$ & $3$ \\
\hline
$X_7$ & $g$ & $A_1$ & $2$ & $1$, $-1$ & $2$ \\
\hline
\hline
$X_6$ & $\mathrm{id}$ & $\varnothing$ & $1$ & $1$, $1$, $1$ & $4$ \\
\hline
$X_6$ & $s$ & $A_1$ & $2$ & $1$, $1$, $-1$ & $3$ \\
\hline
$X_6$ & $rs$ & $A_1^2$ & $2$ & $1$, $-1$, $-1$ & $2$ \\
\hline
$X_6$ & $r^3$ & $A_1$ & $2$ & $1$, $1$, $-1$ & $3$ \\
\hline
$X_6$ & $r^2$ & $A_2$ & $3$ & $1$, $\omega$, $\omega^2$ & $2$ \\
\hline
$X_6$ & $r$ & $A_2 \times A_1$ & $6$ & $\omega$, $\omega^2$, $-1$ & $1$ \\
\hline
\hline
$X_5$ & $\mathrm{id}$ & $\varnothing$ & $1$ & $1$, $1$, $1$, $1$ & $5$ \\
\hline
$X_5$ & $(ab)$ & $A_1$ & $2$ & $1$, $1$, $1$, $-1$ & $4$ \\
\hline
$X_5$ & $(ab)(cd)$ & $A_1^2$ & $2$ & $1$, $1$, $-1$, $-1$ & $3$ \\
\hline
$X_5$ & $(abc)$ & $A_2$ & $3$ & $1$, $1$, $\omega$, $\omega^2$ & $3$ \\
\hline
$X_5$ & $(abcd)$ & $A_3$ & $4$ & $1$, $\ii$, $-1$, $-\ii$ & $2$ \\
\hline
$X_5$ & $(abcde)$ & $A_4$ & $5$ & $\xi_5$, $\xi_5^2$, $\xi_5^3$, $\xi_5^4$ & $1$ \\
\hline
$X_5$ & $(abc)(de)$ & $A_2 \times A_1$ & $6$ & $1$, $\omega$, $\omega^2$, $-1$ & $2$ \\
\hline

\end{tabular}
\end{center}

\caption[]{\label{table1} Types of del Pezzo surfaces of degree $5$ or greater}
\end{table}

\section{Del Pezzo surfaces of degree $2$}

In this section for each cyclic subgroup $\Gamma \subset W(E_7)$ we construct the corresponding del Pezzo surface of degree $2$ over $\F_q$ if it is possible, and show that such surfaces do not exist for other values of $q$. These results are summed up in Theorem \ref{DP2}.

In Table \ref{table2} we collect some facts about conjugacy classes of elements in the Weyl group~$W(E_7)$. This table is based on \cite[Table $10$]{Car72}. The first column is the number of a conjugacy class in order of their appearence in Carter's table. The second column is a Carter graph corresponding to the conjugacy class (see \cite{Car72}). The third column is the order of an element. The fourth column is the collection of eigenvalues of the action of an element on $K_X^{\perp} \subset \Pic(\XX) \otimes \mathbb{Q}$. The fifth column is the invariant Picard number $\rho(\XX)^{\Gamma}$. The last column is the number of the corresponding conjugacy class after the Geiser twist (see Definition \ref{GeiserTwistDef}).

\begin{center}
\begin{longtable}{|c|c|c|c|c|c|}

\hline
Number & Graph & Order & Eigenvalues & $\rho(\XX)^{\Gamma}$ & Geiser \\
\hline
\hline \endhead
1. & $\varnothing$ & $1$ & $1$, $1$, $1$, $1$, $1$, $1$, $1$ & $8$ & 49. \\
\hline
2. & $A_1$ & $2$ & $1$, $1$, $1$, $1$, $1$, $1$, $-1$ & $7$ & 31. \\
\hline
3. & $A_1^2$ & $2$ & $1$, $1$, $1$, $1$, $1$, $-1$, $-1$ & $6$ & 18. \\
\hline
4. & $A_2$ & $3$ & $1$, $1$, $1$, $1$, $1$, $\omega$, $\omega^2$ & $6$ & 53. \\
\hline
5. & $A_1^3$ & $2$ & $1$, $1$, $1$, $1$, $-1$, $-1$, $-1$ & $5$ & 10. \\
\hline
6. & $A_1^3$ & $2$ & $1$, $1$, $1$, $1$, $-1$, $-1$, $-1$ & $5$ & 9. \\
\hline
7. & $A_2 \times A_1$ & $6$ & $1$, $1$, $1$, $1$, $-1$, $\omega$, $\omega^2$ & $5$ & 40. \\
\hline
8. & $A_3$ & $4$ & $1$, $1$, $1$, $1$, $\ii$, $-1$, $-\ii$ & $5$ & 33. \\
\hline
9. & $A_1^4$ & $2$ & $1$, $1$, $1$, $-1$, $-1$, $-1$, $-1$ & $4$ & 6. \\
\hline
10. & $A_1^4$ & $2$ & $1$, $1$, $1$, $-1$, $-1$, $-1$, $-1$ & $4$ & 5. \\
\hline
11. & $A_2 \times A_1^2$ & $6$ & $1$, $1$, $1$, $-1$, $-1$, $\omega$, $\omega^2$ & $4$ & 27. \\
\hline
12. & $A_2^2$ & $3$ & $1$, $1$, $1$, $\omega$, $\omega^2$, $\omega$, $\omega^2$ & $4$ & 55. \\
\hline
13. & $A_3 \times A_1$ & $4$ & $1$, $1$, $1$, $\ii$, $-1$, $-\ii$, $-1$ & $4$ & 22. \\
\hline
14. & $A_3 \times A_1$ & $4$ & $1$, $1$, $1$, $\ii$, $-1$, $-\ii$, $-1$ & $4$ & 21. \\
\hline
15. & $A_4$ & $5$ & $1$, $1$, $1$, $\xi_5$, $\xi_5^2$, $\xi_5^3$, $\xi_5^4$ & $4$ & 54. \\
\hline
16. & $D_4$ & $6$ & $1$, $1$, $1$, $-1$, $-\omega^2$, $-1$, $-\omega$ & $4$ & 19. \\
\hline
17. & $D_4(a_1)$ & $4$ & $1$, $1$, $1$, $\ii$, $-\ii$, $\ii$, $-\ii$ & $4$ & 50. \\
\hline
18. & $A_1^5$ & $2$ & $1$, $1$, $-1$, $-1$, $-1$, $-1$, $-1$ & $3$ & 3. \\
\hline
19. & $A_2 \times A_1^3$ & $6$ & $1$, $1$, $-1$, $-1$, $-1$, $\omega$, $\omega^2$ & $3$ & 16. \\
\hline
20. & $A_2^2 \times A_1$ & $6$ & $1$, $1$, $-1$, $\omega$, $\omega^2$, $\omega$, $\omega^2$ & $3$ & 45. \\
\hline
21. & $A_3 \times A_1^2$ & $4$ & $1$, $1$, $\ii$, $-1$, $-\ii$, $-1$, $-1$ & $3$ & 14. \\
\hline
22. & $A_3 \times A_1^2$ & $4$ & $1$, $1$, $\ii$, $-1$, $-\ii$, $-1$, $-1$ & $3$ & 13. \\
\hline
23. & $A_3 \times A_2$ & $12$ & $1$, $1$, $\ii$, $-1$, $-\ii$, $\omega$, $\omega^2$ & $3$ & 42. \\
\hline
24. & $A_4 \times A_1$ & $10$ & $1$, $1$, $\xi_5$, $\xi_5^2$, $\xi_5^3$, $\xi_5^4$, $-1$ & $3$ & 43. \\
\hline
25. & $A_5$ & $6$ & $1$, $1$, $-\omega^2$, $\omega$, $-1$, $\omega^2$, $-\omega$ & $3$ & 38. \\
\hline
26. & $A_5$ & $6$ & $1$, $1$, $-\omega^2$, $\omega$, $-1$, $\omega^2$, $-\omega$ & $3$ & 37. \\
\hline
27. & $D_4 \times A_1$ & $6$ & $1$, $1$, $-1$, $-\omega^2$, $-1$, $-\omega$, $-1$ & $3$ & 11. \\
\hline
28. & $D_4(a_1) \times A_1$ & $4$ & $1$, $1$, $\ii$, $-\ii$, $\ii$, $-\ii$, $-1$ & $3$ & 35. \\
\hline
29. & $D_5$ & $8$ & $1$, $1$, $-1$, $\xi_8$, $\xi_8^3$, $\xi_8^5$, $\xi_8^7$ & $3$ & 41. \\
\hline
30. & $D_5(a_1)$ & $12$ & $1$, $1$, $\ii$, $-\ii$, $-\omega^2$, $-1$, $-\omega$ & $3$ & 34. \\
\hline
31. & $A_1^6$ & $2$ & $1$, $-1$, $-1$, $-1$, $-1$, $-1$, $-1$ & $2$ & 2. \\
\hline
32. & $A_2^3$ & $3$ & $1$, $\omega$, $\omega^2$, $\omega$, $\omega^2$, $\omega$, $\omega^2$ & $2$ & 60. \\
\hline
33. & $A_3 \times A_1^3$ & $4$ & $1$, $\ii$, $-1$, $-\ii$, $-1$, $-1$, $-1$ & $2$ & 8. \\
\hline
34. & $A_3 \times A_2 \times A_1$ & $12$ & $1$, $\ii$, $-1$, $-\ii$, $\omega$, $\omega^2$, $-1$ & $2$ & 30. \\
\hline
35. & $A_3^2$ & $4$ & $1$, $\ii$, $-1$, $-\ii$, $\ii$, $-1$, $-\ii$ & $2$ & 28. \\
\hline
36. & $A_4 \times A_2$ & $15$ & $1$, $\xi_5$, $\xi_5^2$, $\xi_5^3$, $\xi_5^4$, $\omega$, $\omega^2$ & $2$ & 59. \\
\hline
37. & $A_5 \times A_1$ & $6$ & $1$, $-\omega^2$, $\omega$, $-1$, $\omega^2$, $-\omega$, $-1$ & $2$ & 26. \\
\hline
38. & $A_5 \times A_1$ & $6$ & $1$, $-\omega^2$, $\omega$, $-1$, $\omega^2$, $-\omega$, $-1$ & $2$ & 25. \\
\hline
39. & $A_6$ & $7$ & $1$, $\xi_7$, $\xi_7^2$, $\xi_7^3$, $\xi_7^4$, $\xi_7^5$, $\xi_7^6$ & $2$ & 57. \\
\hline
40. & $D_4 \times A_1^2$ & $6$ & $1$, $-1$, $-\omega^2$, $-1$, $-\omega$, $-1$, $-1$ & $2$ & 7. \\
\hline
41. & $D_5 \times A_1$ & $8$ & $1$, $-1$, $\xi_8$, $\xi_8^3$, $\xi_8^5$, $\xi_8^7$, $-1$ & $2$ & 29. \\
\hline
42. & $D_5(a_1) \times A_1$ & $12$ & $1$, $\ii$, $-\ii$, $-\omega^2$, $-1$, $-\omega$, $-1$ & $2$ & 23. \\
\hline
43. & $D_6$ & $10$ & $1$, $-1$, $-\xi_5^3$, $-\xi_5^4$, $-1$, $-\xi_5$, $-\xi_5^2$ & $2$ & 24. \\
\hline
44. & $D_6(a_1)$ & $8$ & $1$, $\ii$, $-\ii$, $\xi_8$, $\xi_8^3$, $\xi_8^5$, $\xi_8^7$ & $2$ & 52. \\
\hline
45. & $D_6(a_2)$ & $6$ & $1$, $-\omega^2$, $-1$, $-\omega$, $-\omega^2$, $-1$, $-\omega$ & $2$ & 20. \\
\hline
46. & $E_6$ & $12$ & $1$, $\omega$, $\omega^2$, $-\ii\omega$, $-\ii\omega^2$, $\ii\omega$, $\ii\omega^2$ & $2$ & 58. \\
\hline
47. & $E_6(a_1)$ & $9$ & $1$, $\xi_9$, $\xi_9^2$, $\xi_9^4$, $\xi_9^5$, $\xi_9^7$, $\xi_9^8$ & $2$ & 56. \\
\hline
48. & $E_6(a_2)$ & $6$ & $1$, $\omega$, $\omega^2$, $-\omega^2$, $-\omega$, $-\omega^2$, $-\omega$ & $2$ & 51. \\
\hline
49. & $A_1^7$ & $2$ & $-1$, $-1$, $-1$, $-1$, $-1$, $-1$, $-1$ & $1$ & 1. \\
\hline
50. & $A_3^2 \times A_1$ & $4$ & $-1$, $\ii$, $-1$, $-\ii$, $\ii$, $-1$, $-\ii$ & $1$ & 17. \\
\hline
51. & $A_5 \times A_2$ & $6$ & $\omega$, $\omega^2$, $-\omega^2$, $\omega$, $-1$, $\omega^2$, $-\omega$ & $1$ & 48. \\
\hline
52. & $A_7$ & $8$ & $\xi_8$, $\ii$, $\xi_8^3$, $-1$, $\xi_8^5$, $-\ii$, $\xi_8^7$ & $1$ & 44. \\
\hline
53. & $D_4 \times A_1^3$ & $6$ & $-1$, $-1$, $-1$, $-1$, $-\omega^2$, $-1$, $-\omega$ & $1$ & 4. \\
\hline
54. & $D_6 \times A_1$ & $10$ & $-1$, $-1$, $-1$, $-\xi_5^3$, $-\xi_5^4$, $-\xi_5$, $-\xi_5^2$ & $1$ & 15. \\
\hline
55. & $D_6(a_2) \times A_1$ & $6$ & $-1$, $-\omega^2$, $-1$, $-\omega$, $-\omega^2$, $-1$, $-\omega$ & $1$ & 12. \\
\hline
56. & $E_7$ & $18$ & $-1$, $-\xi_9^5$, $-\xi_9^7$, $-\xi_9^8$, $-\xi_9$, $-\xi_8^2$, $-\xi_9^4$ & $1$ & 47. \\
\hline
57. & $E_7(a_1)$ & $14$ & $-\xi_7^4$, $-\xi_7^5$, $-\xi_7^6$, $-1$, $-\xi_7$, $-\xi_7^2$, $-\xi_7^3$ & $1$ & 39. \\
\hline
58. & $E_7(a_2)$ & $12$ & $-\omega^2$, $-1$, $-\omega$, $-\ii\omega$, $-\ii\omega^2$, $\ii\omega$, $\ii\omega^2$ & $1$ & 46. \\
\hline
59. & $E_7(a_3)$ & $30$ & $-\omega^2$, $-1$, $-\omega$, $-\xi_5^3$, $-\xi_5^4$, $-\xi_5$, $-\xi_5^2$ & $1$ & 36. \\
\hline
60. & $E_7(a_4)$ & $6$ & $-\omega^2$, $-1$, $-\omega$, $-\omega^2$, $-\omega$, $-\omega^2$, $-\omega$ & $1$ & 32. \\
\hline
\caption[]{\label{table2} Conjugacy classes of elements in $W(E_7)$}
\end{longtable}
\end{center}

The \textit{Geiser twist} of a del Pezzo surface of degree $2$ is very useful in what follows. It was used in \cite[Section 3]{Tr17} and \cite[Subsection 4.1.2.]{BFL16}. We recall some notions introduced in \cite[Section 3]{Tr17}.

For a del Pezzo surface $X$ of degree $2$ the anticanonical linear system $|-K_X|$ gives a double cover $X \rightarrow \Pro^2_{\F_q}$. This cover defines an involution on $X$ that is called \textit{the Geiser involution}. Therefore we can apply the following well-known proposition (see, for example, \cite[Proposition 4.4]{RT17}).

\begin{proposition}
\label{dPtwist}
Let $X_1$ be a smooth algebraic variety over a finite field $\F_q$ such that a cyclic group $G$ of order $n$ acts on $X_1$ and this action induces a faithful action of~$G$ on the group $\Pic(\XX_1)$. Let $\Gamma_1$ be the image of the Galois group $\Gal\left(\overline{\F}_q / \F_q \right)$ in the group~$\Aut\left(\Pic(\XX_1)\right)$. Let $h$ and $g$ be the generators of $\Gamma_1$ and $G$ respectively.

Then there exists a variety $X_2$ such that the image $\Gamma_2$ of the Galois group $\Gal\left(\overline{\F}_q / \F_q \right)$ in the group $\Aut\left(\Pic(\XX_2)\right) \cong \Aut\left(\Pic(\XX_1)\right)$ is generated by the element~$gh$.
\end{proposition}

Note that in Proposition \ref{dPtwist} one has $\XX_1 \cong \XX_2$. Therefore if $X_1$ is a del Pezzo surface, then $X_2$ is a del Pezzo surface of the same degree.

\begin{definition}[{\cite[Definition 3.2]{Tr17}}]
\label{GeiserTwistDef}
Let $X_1$ be a del Pezzo surface of degree $2$ such that the image $\Gamma_1$ of the Galois group $\Gal\left(\overline{\F}_q / \F_q \right)$ in the group $\Aut\left(\Pic(\XX_1)\right)$ is generated by an element $h$. Then by Proposition \ref{dPtwist} there exists a del Pezzo surface $X_2$ of degree~$2$ such that the image $\Gamma_2$ of the Galois group $\Gal\left(\overline{\F}_q / \F_q \right)$ in the group $\Aut\left(\Pic(\XX_2)\right)$ is generated by the element~$\gamma h$. We say that the surface $X_2$ is the {\it Geiser twist} of the surface~$X_1$.
\end{definition}

Therefore $60$ types of del Pezzo surfaces of degree $2$ are divided into $30$ pairs, and a surface of a given type exists if and only if the Geiser twist of this surface exists.

\begin{remark}[{\cite[Remark 3.3]{Tr17}}]
\label{twisteigenvalues}
Note that the Geiser involution $\gamma$ acts on $K_X^{\perp}$ by multiplying all elements by $-1$. Therefore the eigenvalues of a generator of the group $\Gamma_2$ are the eigenvalues of a generator of the group $\Gamma_1$ multiplied by $-1$. Thus for each type of the group $\Gamma_1$ it is easy to find the type of the corresponding group $\Gamma_2$ (see Table \ref{table2}), except the cases where two types of $\Gamma_2$ have the same collections of eigenvalues.
\end{remark}

We use the following lemma to distinguish types of del Pezzo surfaces with the same collection of eigenvalues.

\begin{lemma}
\label{sameeigenvalues}
If there is a $\Gamma_1$-invariant $(-1)$-curve on a del Pezzo surface $X_1$ of degree $2$ then there are no $\Gamma_2$-invariant $(-1)$-curves on the Geiser twist $X_2$ of $X_1$. In particular, in a pair of del Pezzo surfaces $X_1$ and $X_2$ of degree $2$ such that $X_2$ is the Geiser twist of~$X_1$, no more than one surface is isomorphic to the blowup of a cubic surface at an~$\F_q$-point. 
\end{lemma}

\begin{proof}
If there is a $\Gamma_1$-invariant $(-1)$-curve $E$ on a del Pezzo surface $X_1$ of degree $2$ then one can contract this curve and get a cubic surface $Y$. The set of $(-1)$-curves on $X_1$ consist of $E$, the proper transforms of the $27$ lines on $Y$, and images of these $28$ curves under the action of the Geiser involution. Therefore there are no $\Gamma_1$-invariant pairs of $(-1)$-curves permuted by $\Gamma_1$ and the Geiser involution on $X_1$. Thus there are no \mbox{$\Gamma_2$-invariant} \mbox{$(-1)$-curves} on $X_2$.
\end{proof}

\begin{corollary}
\label{DP2distinguish}
There are six pairs of types of del Pezzo surfaces of degree $2$ that have the same Carter graph: $5$ and $6$, $9$ and $10$, $13$ and $14$, $21$ and $22$, $25$ and $26$, $37$ and $38$ (see Table \ref{table2}). For these types we have the following description.
\begin{itemize}
\item A del Pezzo surface $X_1$ of type $5$ is isomorphic to the blowup of $\Pro^1_{\F_q} \times \Pro^1_{\F_q}$ at three points of degree $2$. The Geiser twist $X_2$ of $X_1$ has type $10$ and is isomorphic to the blowup of a del Pezzo surface of degree $4$ of type $\iota_{abcd}$ (see Table \ref{table4}) at two $\F_q$-points.
\item A del Pezzo surface $X_1$ of type $6$ is isomorphic to the blowup of $\Pro^2_{\F_q}$ at an \mbox{$\F_q$-point} and three points of degree $2$. The Geiser twist $X_2$ of $X_1$ has type $9$ and is isomorphic to the blowup of a del Pezzo surface $Y$ of degree $8$ with $\rho(Y) = 1$ at three points of degree $2$.
\item A del Pezzo surface $X_1$ of type $13$ is isomorphic to the blowup of $\Pro^1_{\F_q} \times \Pro^1_{\F_q}$ at a point of degree $2$ and a point of degree $4$. The Geiser twist $X_2$ of $X_1$ has type $22$ and is isomorphic to the blowup of a del Pezzo surface of degree $4$ of type $(ab)\iota_{acde}$ (see Table \ref{table4}) at two $\F_q$-points.
\item A del Pezzo surface $X_1$ of type $14$ is isomorphic to the blowup of $\Pro^2_{\F_q}$ at an \mbox{$\F_q$-point}, a point of degree $2$ and a point of degree $4$. The Geiser twist $X_2$ of $X_1$ has type $21$ and is isomorphic to the blowup of a del Pezzo surface $Y$ of degree $8$ with $\rho(Y) = 1$ at a point of degree $2$ and a point of degree $4$.
\item A del Pezzo surface $X_1$ of type $25$ is isomorphic to the blowup of $\Pro^1_{\F_q} \times \Pro^1_{\F_q}$ at a point of degree $6$. The Geiser twist $X_2$ of $X_1$ has type $38$ and is isomorphic to the blowup of a cubic surface of type $c_{10}$ (see Table \ref{table3}) at an $\F_q$-point.
\item A del Pezzo surface $X_1$ of type $26$ is isomorphic to the blowup of $\Pro^2_{\F_q}$ at an $\F_q$-point and a point of degree $6$. The Geiser twist $X_2$ of $X_1$ has type $37$ and is isomorphic to the blowup of a del Pezzo surface $Y$ of degree $8$ with $\rho(Y) = 1$ at a point of degree $6$.
\end{itemize}

\end{corollary}

\begin{remark}
\label{nobodycares}
In Table \ref{table2} there are no data describing difference between types of del Pezzo surfaces of degree $2$ with the same Carter graph. One can find these data in \cite[Table~$10$]{Car72} or \cite[Table~$1$]{Ur96} and see that the given in Corollary \ref{DP2distinguish} descriptions satisfy these data.
\end{remark}

\begin{proof}[Proof of Corollary \ref{DP2distinguish}]
The described surfaces of certain types have the corresponding to these types Carter graphs by Lemma \ref{blowup}. The surfaces of types $6$, $10$, $14$, $22$, $26$ and~$38$ are isomorphic to the blowup of certain cubic surfaces at an $\F_q$-point. Therefore by Lemma~\ref{sameeigenvalues} any surface of one of these types cannot be the Geiser twist of an other surface of one of these types.
\end{proof}

We prove Theorem~\ref{DP2} case by case. We start from the types for which a del Pezzo surface of degree $2$ is isomorhic to the blowup of $\Pro^2_{\F_q}$ at a number of points of certain degrees. There are $15$ types of such del Pezzo surfaces parametrised by Young diagrams. Seven of these types are considered in \cite[Subsection 4.2]{BFL16} and three other types are considered in \cite{Tr17}. The remaining five types are considered below.

For types $3$ and $6$ we use the notation of the following remark.

\begin{remark}
\label{cubicnotation}
A del Pezzo surface $X$ of degree $2$ of type $3$ or $6$ is isomorphic to the blowup of a cubic surface $Y$ at an $\F_q$-point not lying on any line. Moreover, there is a morphism $f: Y \rightarrow \Pro^2_{\F_q}$ that contracts six lines on $Y$ to $6$ geometric points $p_1$, \ldots, $p_6$ on $\Pro^2_{\F_q}$. By Theorem \ref{DPlines} the set of lines on $Y$ consists of $E_i = f^{-1}(p_i)$, the proper transforms \mbox{$L_{ij} \sim L - E_i - E_j$} of the lines passing through a pair of points $p_i$ and~$p_j$, and the proper transforms
$$
Q_j \sim 2L + E_j - \sum \limits_{i = 1}^6 E_i
$$
\noindent of the conics passing through five points $p_i$ for $i\neq j$. Note that the group $S_6 \subset W(E_6)$ naturally acts on this set in the following way: for $\sigma \in S_6$ one has $\sigma \left( E_i \right) = E_{\sigma(i)}$, $\sigma \left( L_{ij} \right) = L_{\sigma(i)\sigma(j)}$, and $\sigma \left( Q_i \right) = Q_{\sigma(i)}$. For considered types of surfaces the actions of $\Fr$ on the set of lines on $Y$ coincides with the actions of certain permutations.

\end{remark}

\begin{lemma}
\label{Type3}
A del Pezzo surface of degree $2$ of type $3$ does not exist over $\F_2$, $\F_3$, $\F_4$ and exists over other finite fields.
\end{lemma}

\begin{proof}

If $X$ has type $3$ then the Galois group acts on the set of lines on the corresponding cubic surface $Y$ as the permutation~$(12)(34)$ in the notation of Remark \ref{cubicnotation}. Therefore there are seven lines defined over~$\F_q$: $E_5$, $E_6$, $L_{12}$, $L_{34}$, $L_{56}$, $Q_5$ and~$Q_6$. Also there are two $\F_q$-points $L_{13} \cap L_{24}$ and $L_{14} \cap L_{23}$, that are the points of intersection of the lines not defined over $\F_q$.

The line $L_{56}$ meets each other line defined over $\F_q$, and any other line defined over $\F_q$ meets only two lines defined over $\F_q$. Therefore there are $6q + 3$ points defined over $\F_q$ on the lines $E_5$, $E_6$, $L_{12}$, $L_{34}$, $Q_5$ and~$Q_6$. Each of these lines does not intersect the lines $L_{13}$, $L_{14}$, $L_{23}$ and $L_{24}$. Thus there are at least $6q + 5$ points defined over $\F_q$ lying on the lines on $Y$.

Note that there are $q^2 + 3q + 1$ points defined over $\F_q$ on $Y$. If $q \leqslant 4$ then \mbox{$q^2 + 3q + 1 \leqslant 6q + 5$}, so either the surface $Y$ does not exist over $\F_2$, $\F_3$ and $\F_4$, or it is impossible to blow up $Y$ at an $\F_q$-point not lying on the lines. Therefore a del Pezzo surface of degree $2$ of type $3$ does not exist over $\F_2$, $\F_3$ and $\F_4$.

By Remark \ref{Linespoints}, there are at most $7q + 3$ points defined over $\F_q$ lying on the lines on $Y$. For $q \geqslant 5$ one has $q^2 + 3q + 1 > 7q + 3$. Therefore one can blow up $Y$ at an $\F_q$-point not lying on the lines and get a del Pezzo surface of degree $2$ of type $3$ by Corollary \ref{dPblowup}.

To construct a cubic surface $Y$ for $q \geqslant 5$, one can consider a smooth conic $Q$ on~$\Pro^2_{\F_q}$, choose an $\F_q$-point and two points of degree $2$ on $Q$, and blow up $\Pro^2_{\F_q}$ at these five geometric points. On the obtained surface $Z$ there are $q^2 + 2q + 1$ points defined over~$\F_q$, and $4q + 3$ points defined over $\F_q$ lying on the lines. Therefore we can blow up $Z$ at an $\F_q$-point not lying on the lines and get a required cubic surface $Y$ by Corollary \ref{dPblowup} and Lemma \ref{blowup}.

\end{proof}

\begin{lemma}
\label{Type6}
A del Pezzo surface of degree $2$ of type $6$ does not exist over $\F_2$ and exists over other finite fields.
\end{lemma}

\begin{proof}

If $X$ has type $6$ then the Galois group acts on the set of lines on the corresponding cubic surface $Y$ as the permutation~$(12)(34)(56)$.

Over any finite field $\F_q$ the cubic surface $Y$ can be constructed in the following way. Let $p_1$, $p_2$ and $p_3$, $p_4$ be two pairs of conjugate geometric points defined over $\F_{q^2}$ in general position on $\Pro^2_{\F_q}$. Consider two smooth conjugate conics $Q_1$ and $Q_2$ defined over $\F_{q^2}$ passing through $p_1$, $p_2$, $p_3$ and $p_4$ (such conics always exist since three reduced conics are defined over $\F_q$). Let $q_1$ be an $\F_{q^2}$-point on $Q_1$ and $q_2 \in Q_2$ be the conjugate of $q_1$. We show that the points $p_i$ and $q_j$ are in general position.

Any line passing through two points $p_i$ and $p_j$ does not contain other points lying on $Q_1$ or $Q_2$. Therefore such line passes through exactly two points from the set $\{p_i, q_j\}$. If a point $p_i$ lie on a line $L$ passing through $q_1$ and $q_2$ then the conjugate of $p_i$ lie on this line since $q_1$ and $q_2$ are conjugate. Hence $L$ meets $Q_1$ at three points, that is impossible.

Thus the points $p_i$ and $q_j$ are in general position. The blowup of $Y \rightarrow \Pro^2_{\F_q}$ at these points is a required cubic surface by Theorem \ref{GenPos} and Lemma \ref{blowup}.

There are three lines on $Y$ defined over $\F_q$: $L_{12}$, $L_{34}$, and $L_{56}$. Note that the lines $L_{12}$, $L_{34}$, and $L_{56}$ meet each other. Also there are six \mbox{$\F_q$-points} $L_{13} \cap L_{24}$, $L_{14} \cap L_{23}$, $L_{15} \cap L_{26}$, $L_{16} \cap L_{25}$, $L_{35} \cap L_{46}$, and $L_{36} \cap L_{45}$, that are the points of intersection of the lines not defined over $\F_q$. Thus, by Remark \ref{Linespoints}, there are at most $3q + 7$ points defined over $\F_q$ lying on the lines on $Y$.

Note that there are $q^2 + q + 1$ points defined over $\F_q$ on $Y$. For $q \geqslant 4$ one \mbox{has $q^2 + q + 1 > 3q + 7$}. Therefore one can blow up $Y$ at an $\F_q$-point not lying on the lines and get a del Pezzo surface of degree $2$ of type $6$ by Corollary \ref{dPblowup}.

For $q = 3$ one can check that the seven points
$$
\left( 1 : 0 : 1 + \ii \right), \qquad \left( 1 : 0 : 1 - \ii \right), \qquad \left( 0 : 1 : 1 + \ii \right), \qquad \left( 0 : 1 : 1 - \ii \right),
$$
$$
\left( 1 : 1 + \ii : 0 \right), \qquad \left( 1 : 1 - \ii : 0 \right) \qquad \left( 1 : -1 : 1 \right),
$$
\noindent are in general position.

A del Pezzo surface of degree $2$ of type $6$ over $\F_2$ does not exist, since there are no seven $\F_4$-points on $\Pro^2_{\F_4}$ in general position by \cite[Proposition 4.5]{BFL16}.

\end{proof}

Now we consider types $11$, $14$ and $23$.

\begin{lemma}
\label{Type11}
A del Pezzo surface of degree $2$ of type $11$ exists over all finite fields.
\end{lemma}

\begin{proof}

A del Pezzo surface of type $11$ is the blowup of a $\Pro^2_{\F_q}$ at two points of degree $2$ and a point of degree $3$.

Let $r$, $q_1$, $q_2$, $q_3$ be geometric points on $\Pro^2_{\F_q}$ in general position such that $r$ is an $\F_q$-point, $q_i$ are conjugate points defined over $\F_{q^3}$. Consider two smooth conics $Q_1$ and $Q_2$ defined over $\F_q$ and passing through these four points. Let $L_1$ and $L_2$ be two conjugate lines defined over $\F_{q^2}$ and passing through $r$, and $p_{ij}$ be the points of intersection of $L_i$ and $Q_j$ that differ from $r$. Then the points $p_{ij}$ are defined over $\F_{q^2}$. We show that the points $p_{ij}$ and $q_k$ are in general position.

Any line passing through two points $q_i$ and $q_j$ does not contain other points lying on $Q_1$ or $Q_2$. Therefore such line passing through exactly two points from the set $\{p_{ij}, q_k\}$. If a point $q_k$ lie on a line passing through two points from the set $\{p_{ij}\}$ then the points $\Fr^2q_k$ and $\Fr^4q_k$ lie on this line since $\Fr^2p_{ij} = p_{ij}$ for any $i$ and $j$. But this is impossible since any line meets $Q_1$ at $2$ or $1$ point. If a line $L$ passing through three points from the set $\{p_{ij}\}$ then it has two common points with $L_1$ or $L_2$. That is impossible.

If a conic $Q$ passes through the points $q_1$, $q_2$, $q_3$, and three points from the set $\{p_{ij}\}$ then it has five common points with $Q_1$ or $Q_2$. That is impossible. If a conic $Q$ passes through the six points $p_{11}$, $p_{12}$, $p_{21}$, $p_{22}$, $q_i$, and $q_j$, then the conics $Q$ and $\Fr Q$ have at least five common points. Therefore one has $Q = \Fr Q$ and this conic contains the points $q_1$, $q_2$, $q_3$ and have five common points with $Q_1$ and $Q_2$. That is impossible.

Thus the points $p_{ij}$ and $q_k$ are in general position. The blowup of $\Pro^2_{\F_q}$ at these points is a del Pezzo surface of degree $2$ of type $11$ by Theorem \ref{GenPos} and Lemma \ref{blowup}.

\end{proof}

\begin{lemma}
\label{Type14}
A del Pezzo surface of degree $2$ of type $14$ does not exist over $\F_2$ and exists over other finite fields.
\end{lemma}

\begin{proof}

A del Pezzo surface of type $14$ is the blowup of a $\Pro^2_{\F_q}$ at an $\F_q$-point, a point of degree $2$ and a point of degree $4$.

Let us show that there are no a point of degree $2$ and a point of degree $4$ in general position on $\Pro^2_{\F_2}$. Assume that there exist such points in general position and $p_1$, $p_2$, $q_1$, $q_2$, $q_3$, $q_4$ are the corresponding geometric points. Then a conic $Q_i$ passing through $p_i$, $q_1$, $q_2$, $q_3$, $q_4$ is defined over $\F_4$ and not defined over $\F_2$. There are only two conjugate conics defined over $\F_4$ passing through four points since there is only one point of degree~$2$ on~$\Pro^1_{\F_2}$. Let $L_{ij}$ be a line passing through $q_i$ and $q_j$. Then the conics $L_{12} \cup L_{34}$ and $L_{14} \cup L_{23}$ are conjugate conics defined over $\F_4$ passing through $q_1$, $q_2$, $q_3$, $q_4$. Therefore $Q_1$ and $Q_2$ coincide with these conics, and points $q_1$ and $q_2$ lie on the union of lines $L_{ij}$. Thus the points $p_1$, $p_2$, $q_1$, $q_2$, $q_3$, $q_4$ are not in general position. We have a contradiction.

Now assume that $q \geqslant 3$. Let $q_1$, $q_2$, $q_3$, and $q_4$ be conjugate geometric points defined over~$\F_{q^4}$ in general position on $\Pro^2_{\F_q}$. Consider two smooth conjugate conics $Q_1$ and $Q_2$ defined over $\F_{q^2}$ passing through $q_1$, $q_2$, $q_3$ and $q_4$. Let $p_1$ be an $\F_{q^2}$-point on $Q_1$ and~$p_2 \in Q_2$ be the conjugate of $p_1$. We show that the points $p_i$ and $q_j$ are in general position.

Any line passing through two points $q_i$ and $q_j$ does not contain other points lying on~$Q_1$ or $Q_2$. Therefore such line passes through exactly two points from the set $\{p_i, q_j\}$. If a point $q_i$ lie on a line passing through $p_1$ and $p_2$ then the points $\Fr q_i$, $\Fr^2q_i$ and $\Fr^3q_i$ lie on this line since $\Fr p_1 = p_2$ and $\Fr p_2 = p_1$. But this is impossible.

Thus the points $p_i$ and $q_j$ are in general position. The blowup of $Y \rightarrow \Pro^2_{\F_q}$ at these points is a cubic surface by Theorem \ref{GenPos}. If we blow up an $\F_q$-point not lying on the lines on $Y$ then we get a del Pezzo surface of degree $2$ of type $14$, by Corollary \ref{dPblowup} and Lemma~\ref{blowup}. Such point exists since there are $q^2 + q + 1$ points defined over $\F_q$ on $Y$, and no more than $q + 2$ from those points lie on the lines (in the notation of Remark \ref{cubicnotation} these points are $\F_q$-points on $L_{12}$ and the $\F_q$-point $L_{35} \cap L_{46}$).

\end{proof}

\begin{lemma}
\label{Type23}
A del Pezzo surface of degree $2$ of type $23$ exists over all finite fields.
\end{lemma}

\begin{proof}

A del Pezzo surface of type $23$ is the blowup of $\Pro^2_{\F_q}$ at a point of degree $3$ and a point of degree $4$.

Let $q_1$, $q_2$, $q_3$ be conjugate geometric points on $\Pro^2_{\F_q}$ not lying on a line and defined over~$\F_{q^3}$. There are $q^2 + q + 1$ lines on $\Pro^2_{\F_q}$ and $q^2 + q + 1$ conics passing through $q_1$, $q_2$ and $q_3$. Each of those curves contains $\frac{q^4 - q^2}{4}$ points of degree $4$. There are $\frac{q^8 - q^2}{4}$ points of degree $4$ on $\Pro^2_{\F_q}$. For any $q \geqslant 2$ one can check that
$$
\frac{q^8 - q^2}{4} > 2 \cdot \frac{q^4 - q^2}{4} \cdot (q^2 + q + 1).
$$
Therefore we can find a point of degree $4$ not lying on a line or a conic passing through $q_1$, $q_2$ and $q_3$. Denote the corresponding geometric points by $p_1$, $p_2$, $p_3$ and $p_4$. We show that the points $p_i$ and $q_j$ are in general position.

Assume that points $p_i$, $p_j$, and $q_k$ lie on a line. Then the points $\Fr^4p_i = p_i$, $\Fr^4p_j = p_j$, and $\Fr^4q_k$ lie on the same line. Therefore the points $q_1$, $q_2$, and $q_3$ lie on a line. We have a contradiction. The same arguments show that three points $p_i$, $q_j$, and $q_k$ can not lie on a line.

If a conic passes through six points $p_i$, $p_j$, $p_k$, $q_1$, $q_2$, and $q_3$ then this conic passes through the points $p_1$, $p_2$, $p_3$, $p_4$ since the sets $\{ \Fr p_i, \Fr p_j, \Fr p_k\}$ and $\{p_i, p_j, p_k\}$ have two common points. If a conic passes through the six points $p_1$, $p_2$, $p_3$, $p_4$, $q_i$, $q_j$ then it passes through the points $q_1$, $q_2$, $q_3$ since either $\Fr q_i = q_j$ or $\Fr q_j = q_i$. All these cases are impossible since the points $p_1$, $p_2$, $p_3$, $p_4$ do not lie on a conic passing through the points $q_1$, $q_2$, and~$q_3$.

Thus the points $p_j$ and $q_j$ are in general position. The blowup of $\Pro^2_{\F_q}$ at these points is a del Pezzo surface of degree $2$ of type $23$ by Theorem \ref{GenPos} and Lemma \ref{blowup}.

\end{proof}

The next five types of del Pezzo surface of degree $2$ can be obtained by blowing up minimal del Pezzo surfaces of degree $4$ of certain types at two $\F_q$-points. The method that we use for these types is very closely related to the method used in \cite[Lemma~3.12]{Tr17}, where there is considered the blowup of the sixth type of a minimal del Pezzo surface of degree $4$ at two $\F_q$-points. 

\begin{lemma}
\label{Type10}
A del Pezzo surface of degree $2$ of type $10$ does not exist over $\F_2$, $\F_3$, $\F_4$, $\F_5$ and exists over other finite fields.
\end{lemma}

\begin{proof}

A del Pezzo surface of degree $2$ of type $10$ is the blowup of a minimal del Pezzo surface $S$ of degree $4$ at two $\F_q$-points. The surface $S$ admits a conic bundle structure with degenerate fibres over four $\F_q$-points. Such del Pezzo surface does not exist over $\F_2$ and~$\F_3$, and exists over other finite fields by \cite[Theorem 3.2]{Ry05} and \cite[Theorem~2.5]{Tr17}. Assume that~$q \geqslant 4$. The surface $S$ admits two structures of conic bundles and each of the $16$ lines is a component of a singular fibre of one of these conic bundles. These lines form eight $\langle \Fr \rangle$-orbits, each consisting of $2$ curves. Therefore there are eight $\F_q$-points on the lines, that are the points of intersection of conjugate lines. But by equation \eqref{Fpoints} there are $q^2 - 2q + 1$ points defined over $\F_q$ on~$S$. Thus for $q \geqslant 4$ there is an $\F_q$-point $P$ on~$S$ not lying on the lines. Let $f:\widetilde{S} \rightarrow S$ be the blowup of $S$ at $P$. By Corollary \ref{dPblowup}, the surface~$\widetilde{S}$ is a cubic surface of type $(c_3)$ (see Table \ref{table3}). There are three lines on $\widetilde{S}$ defined over $\F_q$: the exceptional divisor $E = f^{-1}(P)$, and the proper transforms $C_1$ and $C_2$ of the fibres passing through~$P$ of the two conic bundles on $S$.

Now we show that all other $\langle \Fr \rangle$-orbits of lines consist of $2$ lines. Let $H$ be a line on $\widetilde{S}$ that differs from $E$, $C_1$ and $C_2$. If $H \cdot E = 0$ then $f(L)$ is a $(-1)$-curve and the orbit of this curve consists of $2$ curves. Assume that $H \cdot E = 1$. Then $C_1 \cdot H = C_2 \cdot H = 0$ since $E + C_1 + C_2 \sim -K_{\widetilde{S}}$. It means that $f(H)$ is a section of any conic bundle on $S$. For any singular fibre this section must meet one component $D_1$ of this fibre at a point, and for the other component $D_2$ of this fibre $f(H) \cdot D_2 = 0$. But we have $\Fr D_1 = D_2$, therefore $\Fr f(H) \cdot D_2 = f(H) \cdot D_1 = 1$. Thus $\Fr f(H) \ne f(H)$ and the orbit of $H$ consists of $2$ lines.

Therefore on $\widetilde{S}$ there are twelve $\F_q$-points that are the points of intersection of conjugate lines, and three meeting each other lines $E$, $C_1$, and $C_2$ defined over $\F_q$. By Remark \ref{Linespoints}, there are $3q + 12 - \epsilon$ or $3q + 13 - \epsilon$ points defined over $\F_q$ lying on the lines on $Y$, \mbox{where $\epsilon \leqslant 12$} is the number of Eckardt points among the points of intersection of conjugate lines.

By equation \eqref{Fpoints}, there are $q^2 - q + 1$ points defined over $\F_q$ on $\widetilde{S}$. For $q \geqslant 7$ one has \mbox{$q^2 - q + 1 > 3q + 13$}. Therefore one can blow up $\widetilde{S}$ at an $\F_q$-point not lying on the lines and get a del Pezzo surface of degree $2$ of type $10$ by Corollary \ref{dPblowup} and Lemma \ref{blowup}.

By Lemma \ref{Eckardtlines}, for $q = 5$ there are at most $2$ Eckardt points on a line. Therefore if~$E$, $C_1$, and $C_2$ meet at an Eckardt point then there are $3q + 1$ points defined over $\F_q$ on their union and~$\epsilon \leqslant 3$. In this case there are at least $25$ $\F_q$-points lying on the lines on~$\widetilde{S}$. Otherwise, there are $3q$ points defined over $\F_q$ on $E \cup C_1 \cup C_2$ and~$\epsilon \leqslant 6$. In this case there are at least $21$ $\F_q$-points lying on the lines on $\widetilde{S}$. In the both cases there is no an~$\F_q$-point not lying on the lines on $\widetilde{S}$.

By Lemma \ref{Eckardtlines}, for $q = 4$ there are either $0$, or $1$, or $5$ Eckardt points on a line. Therefore if $E$, $C_1$ and $C_2$ meet at an Eckardt point then there are $3q + 1$ points defined over $\F_q$ on their union and~$\epsilon \leqslant 12$. In this case there are at least $13$ $\F_q$-points lying on the lines on~$\widetilde{S}$. Otherwise, there are $3q$ points defined over $\F_q$ on $E \cup C_1 \cup C_2$ and~$\epsilon \leqslant 3$. In this case there are at least $21$ $\F_q$-points lying on the lines on~$\widetilde{S}$. In the both cases there is no an $\F_q$-point not lying on the lines on $\widetilde{S}$.

For $q \leqslant 3$ one has $q^2 - q + 1 < 3q$. Therefore a del Pezzo surface of degree $2$ of type $10$ can not be obtained over $\F_2$, $\F_3$, $\F_4$, and $\F_5$.

\end{proof}

\begin{lemma}
\label{Type16}
A del Pezzo surface of degree $2$ of type $16$ exists over all finite fields.
\end{lemma}

\begin{proof}

A del Pezzo surface of degree $2$ of type $16$ is the blowup of a minimal del Pezzo surface $S$ of degree $4$ at two $\F_q$-points. The surface $S$ admits a conic bundle structure with degenerate fibres over an $\F_q$-point and a point of degree $3$. Such del Pezzo surface exists over all finite fields by \cite[Theorem 3.2]{Ry05}. The surface $S$ admits two structures of conic bundles and each of the $16$ lines is a component of a singular fibre of one of these conic bundles. The orbits of $\langle \Fr \rangle$ on the set of the lines on $S$ have cardinalities $2$, $2$, $6$, and~$6$. Therefore there are two $\F_q$-points on the lines, that are the points of intersection of the lines in the orbits of length two. But by equation \eqref{Fpoints} there are $q^2 + q + 1$ points defined over $\F_q$ on~$S$. Thus there is an $\F_q$-point $P$ on $S$ not lying on the lines. Let $f:\widetilde{S} \rightarrow S$ be the blowup of $S$ at $P$. By Corollary \ref{dPblowup}, the surface $\widetilde{S}$ is a cubic surface. There are three lines on $\widetilde{S}$ defined over $\F_q$: the exceptional divisor $E = f^{-1}(P)$, and the proper transforms $C_1$ and $C_2$ of the fibres passing through $P$ of the two conic bundles.

As in Lemma \ref{Type10} one can check that the $\langle F \rangle$-orbits of lines meeting $E$ at a point have cardinalities $1$, $1$, $2$, and $6$.

Therefore on $\widetilde{S}$ there are three $\F_q$-points that are the points of intersection of pairs of conjugate lines, and three meeting each other lines $E$, $C_1$, and $C_2$ defined over $\F_q$. By Remark \ref{Linespoints}, there are at most $3q + 4$ points defined over $\F_q$ lying on the lines on $Y$.

By equation \eqref{Fpoints}, there are $q^2 + 2q + 1$ points defined over $\F_q$ on $\widetilde{S}$. For $q \geqslant 3$ one \mbox{has $q^2 + 2q + 1 > 3q + 4$}. Therefore one can blow up $\widetilde{S}$ at an $\F_q$-point not lying on the lines and get a del Pezzo surface of degree $2$ of type $10$ by Corollary \ref{dPblowup} and Lemma \ref{blowup}.

For $q = 2$ we should more carefully choose a point $P$ on $S$.

Let $\pi_1: S \rightarrow \Pro^1_{\F_2}$ and $\pi_2: S \rightarrow \Pro^1_{\F_2}$ be the two conic bundles structures, $s_1$ and $s_2$ be $\F_2$-points on a singular fibres of $\pi_1$ and $\pi_2$ respectively, $F_1$ be a fibre of $\pi_1$ containing~$s_2$ and $F_2$ be a fibre of $\pi_2$ containing $s_1$. If $F_1$ and $F_2$ have a common $\F_2$-point $P$, then one can blow up this point and get a cubic surface $\widetilde{S}$, such that at most eight $\F_2$-points lying on the lines on $\widetilde{S}$, since the two points of intersection of conjugate lines lie on lines defined over $\F_2$.

If $F_1$ and $F_2$ do not have a common $\F_2$-point, then the other smooth fibre $\widetilde{F}_2$ of $\pi_2$ defined over $\F_2$ transversally intersects $F_1$ at two $\F_2$-points. One can blow up any of these points and get a cubic surface $\widetilde{S}$, such that at most eight $\F_2$-points lying on the lines on~$\widetilde{S}$, since one of the points of intersection of conjugate lines lies on a line defined over $\F_2$ and the three lines defined over $\F_2$ do not have a common Eckardt point.

In the both cases there are nine $\F_2$-points on $\widetilde{S}$ and one can blow up $\widetilde{S}$ at an $\F_2$-point not lying on the lines and get a del Pezzo surface of degree $2$ of type~$10$.

\end{proof}

Now we consider del Pezzo surfaces of degree $2$ of types $22$, $29$, and $30$. These surfaces are the blowups of del Pezzo surfaces of degree $4$ with the Picard number $1$ at \mbox{two $\F_q$-points}. The blowup of such a surface at an $\F_q$-point not lying on the lines is a cubic surface admitting a structure of a conic bundle. These cubic surfaces exist over any field by \cite[Theorem~3.2]{Ry05}, and have types $(c_{19})$, $(c_{20})$, and $(c_{24})$ (see Table \ref{table3}) by Lemma \ref{blowup}. 

\begin{lemma}
\label{Types222930}
A del Pezzo surface of degree $2$ of type $22$ does not exist over $\F_2$ and exists over other finite fields.

Del Pezzo surfaces of degree $2$ of types $29$ and $30$ exist over all finite fields.
\end{lemma}

\begin{proof}

From \cite[Table 1]{Man74} one can see, that the $\langle \Fr \rangle$-orbits of the lines of length at most $3$ have cardinalities $1$, $2$, $2$, and $2$ for a cubic surface of type $(c_{19})$, have cardinalities $1$ and $2$ for a cubic surface of type $(c_{20})$, and have cardinality $1$ for a cubic surface of type~$(c_{24})$. Therefore by Remark~\ref{Linespoints} there are at most $q + 4$, $q + 2$ and $q + 1$ points defined over $\F_q$ and lying on the lines for cubic surfaces of types $(c_{19})$, $(c_{20})$ and $(c_{24})$ respectively. By equation~\eqref{Fpoints}, there are $q^2 - q + 1$, $q^2 + q + 1$ and $q^2 + 2q + 1$ points defined over $\F_q$ on these types of cubic surfaces respectively.

Note that for $q \geqslant 4$ one has $q^2 - q + 1 > q + 4$, and for $q \geqslant 2$ one has $q^2 + q + 1 > q + 2$, and $q^2 + 2q + 1 > q + 1$. Therefore one can blow up a cubic surface of type $(c_{19})$ for $q \geqslant 4$ or a cubic surface of type $(c_{20})$ or $(c_{24})$ over any finite field at an $\F_q$-point not lying on the lines and get a del Pezzo surface of degree $2$ of types $22$, $29$ or $30$ respectively by Corollary \ref{dPblowup} and Lemma \ref{blowup}.

For $q = 3$ a cubic surface of type $(c_{19})$ has a structure of a conic bundle with three degenerate fibres and one smooth fibre $F$ defined over $\F_3$. These singular fibres consist of pairs of conjugate lines defined over $\F_9$, and the line defined over $\F_3$ is a bisection of this conic bundle (see the proof of \cite[Theorem 4]{Isk79}). Therefore there are at least two~$\F_3$-points on $F$ not lying on the lines. One can blow up the cubic surface at one of these~$\F_3$-points and get a del Pezzo surface of degree $2$ of type $22$.

For $q = 2$ a cubic surface of type $(c_{19})$ has a structure of a conic bundle with three degenerate fibers defined over $\F_2$. Therefore there is no an $\F_2$-point not lying on the lines. Thus a del Pezzo surface of degree $2$ of type $22$ can not be obtained over $\F_2$.

\end{proof}

The remaining case is a del Pezzo surface of degree $2$ of type $38$. This surface is the blowup of a minimal cubic surface of type $(c_{10})$ at an $\F_q$-point not lying on the lines.

\begin{lemma}
\label{Type38}
A del Pezzo surface of degree $2$ of type $38$ does not exist over $\F_2$ and exists over other finite fields.
\end{lemma}

\begin{proof}

A del Pezzo surface of type $38$ is the blowup of a cubic surface of type $(c_{10})$ at \mbox{an $\F_q$-point}. From \cite[Table 1]{Man74} one can see that the $\langle \Fr \rangle$-orbits of lines have cardinalities $3$, $6$, $6$, $6$ and $6$ for a cubic surface of type $(c_{10})$. Therefore there is at most one $\F_q$-point lying on the lines for this cubic surface. By equation \eqref{Fpoints} there are $q^2 - q + 1$ points defined over $\F_q$ on this cubic surface. One has $q^2 - q + 1 > 1$ for any $q$.

For $q \geqslant 3$ a cubic surface of type $(c_{10})$ exists by \cite[Proposition 6.1]{RT17}. Therefore one can blow up this surface at an $\F_q$-point not lying on the lines and get a del Pezzo surface of degree $2$ of type $38$ by Corollary \ref{dPblowup} and Lemma \ref{blowup}.

For $q = 2$ a cubic surface of type $(c_{10})$ does not exist by \cite[Proposition 6.5]{RT17}. Thus a del Pezzo surface of degree $2$ of type $38$ can not be obtained over $\F_2$.

\end{proof}

The blowups of minimal cubic surfaces at an $\F_q$-point of other types $(c_{11})$--$(c_{14})$ are considered in \cite[Lemma 3.13]{Tr17}. But for the three types $(c_{12})$--$(c_{14})$ of these surfaces there are some restrictions on $q$ coming from \cite[Theorem 1.2]{RT17}. In Section $5$ we remove these restrictions for the types $(c_{12})$ and $(c_{13})$.

Now we prove Theorem \ref{DP2}.

\begin{proof}[Proof of Theorem \ref{DP2}]

Del Pezzo surfaces of degree $2$ of types that are the Geiser twists (see Definition \ref{GeiserTwistDef}) of each other exist over the same finite fields. For each of the $30$ pairs of types we give a reference, where one of these types is considered.

\begin{itemize}

\item The types $1$ and $49$ are considered in \cite[Subsection 4.2, case $a = 8$]{BFL16}.
\item The types $2$ and $31$ are considered in \cite[Subsection 4.2, case $a = 6$]{BFL16}.
\item The types $3$ and $18$ are considered in Lemma \ref{Type3}.
\item The types $4$ and $53$ are considered in \cite[Subsection 4.2, case $a = 5$]{BFL16}.
\item The types $5$ and $10$ are considered in Lemma \ref{Type10}.
\item The types $6$ and $9$ are considered in Lemma \ref{Type6}.
\item The types $7$ and $40$ are considered in \cite[Proposition 2.17 and Lemma 3.5]{Tr17}.
\item The types $8$ and $33$ are considered in \cite[Subsection 4.2, case $a = 4$]{BFL16}.
\item The types $11$ and $27$ are considered in Lemma \ref{Type11}.
\item The types $12$ and $55$ are considered in \cite[Lemma 3.10]{Tr17}.
\item The types $13$ and $22$ are considered in Lemma \ref{Types222930}.
\item The types $14$ and $21$ are considered in Lemma \ref{Type14}.
\item The types $15$ and $54$ are considered in \cite[Subsection 4.2, case $a = 3$]{BFL16}.
\item The types $16$ and $19$ are considered in Lemma \ref{Type16}.
\item The types $17$ and $50$ are considered in \cite[Lemma 3.12]{Tr17}.
\item The types $20$ and $45$ are considered in \cite[Proposition 2.17]{Tr17}.
\item The types $23$ and $42$ are considered in Lemma \ref{Type23}.
\item The types $24$ and $43$ are considered in \cite[Subsection 4.2, case $a = 1$]{BFL16}.
\item The types $25$ and $38$ are considered in Lemma \ref{Type38}.
\item The types $26$ and $37$ are considered in \cite[Subsection 4.2, case $a = 2$]{BFL16}.
\item The types $28$ and $35$ are considered in \cite[Proposition 2.17]{Tr17}.
\item The types $29$ and $41$ are considered in Lemma \ref{Types222930}.
\item The types $30$ and $34$ are considered in Lemma \ref{Types222930}.
\item The types $32$ and $60$ are considered in \cite[Lemma 3.13]{Tr17}.
\item The types $36$ and $59$ are considered in \cite[Lemma 3.11]{Tr17}.
\item The types $39$ and $57$ are considered in \cite[Lemma 3.10]{Tr17}.
\item The types $44$ and $52$ are considered in \cite[Proposition 2.17]{Tr17}.
\item The types $46$ and $58$ are considered in \cite[Lemma 3.13]{Tr17} and Lemma \ref{TypesC12C13}.
\item The types $47$ and $56$ are considered in \cite[Lemma 3.13]{Tr17} and \cite{RT17}.
\item The types $48$ and $51$ are considered in \cite[Lemma 3.13]{Tr17} and Lemma \ref{TypesC12C13}.
\end{itemize}

\end{proof}

\section{Del Pezzo surfaces of degree $3$}

In this section for each cyclic subgroup $\Gamma \subset W(E_6)$ we construct the corresponding cubic surface (that is a del Pezzo surface of degree $3$) over $\F_q$ if it is possible, and show that such surfaces do not exist for other values of $q$. As a result we prove Theorem~\ref{DP3}.

In Table \ref{table3} we collect some facts about conjugacy classes of elements in the Weyl group~$W(E_6)$. This table based on the tables in \cite{SD67}, \cite{Man74}, and \cite{Car72}. The first column is a type according to \cite{SD67}. The second column is the Carter graph corresponding to the conjugacy class (see \cite{Car72}). The third column is the order of element. The fourth column is the collection of eigenvalues of the action of element on $K_X^{\perp} \subset \Pic(\XX) \otimes \mathbb{Q}$. The fifth column is the invariant Picard number $\rho(\XX)^{\Gamma}$. The last column is the type of the corresponding conjugacy class (see Table \ref{table2}) in $W(E_7)$ after blowing up a cubic surface at an $\F_q$-point.

\begin{table}

\begin{center}
\begin{tabular}{|c|c|c|c|c|c|}
\hline
Type & Graph & Order & Eigenvalues & $\rho(\XX)^{\Gamma}$ & Blowup \\
\hline
$c_1$ & $\varnothing$ & $1$ & $1$, $1$, $1$, $1$, $1$, $1$ & $7$ & 1. \\
\hline
$c_2$ & $A_1^2$ & $2$ & $-1$, $-1$, $1$, $1$, $1$, $1$ & $5$ & 3. \\
\hline
$c_3$ & $A_1^4$ & $2$ & $-1$, $-1$, $-1$, $-1$, $1$, $1$ & $3$ & 10. \\
\hline
$c_4$ & $D_4(a_1)$ & $4$ & $\ii$, $\ii$, $-\ii$, $-\ii$, $1$, $1$ & $3$ & 17. \\
\hline
$c_5$ & $A_3 \times A_1$ & $4$ & $\ii$, $-\ii$, $-1$, $-1$, $1$, $1$ & $3$ & 14. \\
\hline
$c_6$ & $A_2$ & $3$ & $\omega$, $\omega^2$, $1$, $1$, $1$, $1$ & $5$ & 4. \\
\hline
$c_7$ & $D_4$ & $6$ & $-\omega$, $-\omega^2$, $-1$, $-1$, $1$, $1$ & $3$ & 16. \\
\hline
$c_8$ & $A_2 \times A_1^2$ & $6$ & $\omega$, $\omega^2$, $-1$, $-1$, $1$, $1$ & $3$ & 11. \\
\hline
$c_9$ & $A_2^2$ & $3$ & $\omega$, $\omega$, $\omega^2$, $\omega^2$, $1$, $1$ & $3$ & 12. \\
\hline
$c_{10}$ & $A_5 \times A_1$ & $6$ & $-\omega$, $-\omega^2$, $\omega$, $\omega^2$, $-1$, $-1$ & $1$ & 38. \\
\hline
$c_{11}$ & $A_2^3$ & $3$ & $\omega$, $\omega$, $\omega$, $\omega^2$, $\omega^2$, $\omega^2$ & $1$ & 32. \\
\hline
$c_{12}$ & $E_6(a_2)$ & $6$ & $-\omega$, $-\omega$, $-\omega^2$, $-\omega^2$, $\omega$, $\omega^2$ & $1$ & 48. \\
\hline
$c_{13}$ & $E_6$ & $12$ & $\ii\omega$, $\ii\omega^2$, $-\ii\omega$, $-\ii\omega^2$, $\omega$, $\omega^2$ & $1$ & 46. \\
\hline
$c_{14}$ & $E_6(a_1)$ & $9$ & $\xi_9$, $\xi_9^2$, $\xi_9^4$, $\xi_9^5$, $\xi_9^7$, $\xi_9^8$ & $1$ & 47. \\
\hline
$c_{15}$ & $A_4$ & $5$ & $\xi_5$, $\xi_5^2$, $\xi_5^3$, $\xi_5^4$, $1$, $1$ & $3$ & 15. \\
\hline
$c_{16}$ & $A_1$ & $2$ & $-1$, $1$, $1$, $1$, $1$, $1$ & $6$ & 2. \\
\hline
$c_{17}$ & $A_1^3$ & $2$ & $-1$, $-1$, $-1$, $1$, $1$, $1$ & $4$ & 6. \\
\hline
$c_{18}$ & $A_3$ & $4$ & $\ii$, $-\ii$, $-1$, $1$, $1$, $1$ & $4$ & 8. \\
\hline
$c_{19}$ & $A_3 \times A_1^2$ & $4$ & $\ii$, $-\ii$, $-1$, $-1$, $-1$, $1$ & $2$ & 22. \\
\hline
$c_{20}$ & $D_5$ & $8$ & $\xi_8$, $\xi_8^3$, $\xi_8^5$, $\xi_8^7$, $-1$, $1$ & $2$ & 29. \\
\hline
$c_{21}$ & $A_2 \times A_1$ & $6$ & $\omega$, $\omega^2$, $-1$, $1$, $1$, $1$ & $4$ & 7. \\
\hline
$c_{22}$ & $A_2^2 \times A_1$ & $6$ & $\omega$, $\omega$, $\omega^2$, $\omega^2$, $-1$, $1$ & $2$ & 20. \\
\hline
$c_{23}$ & $A_5$ & $6$ & $-\omega$, $-\omega^2$, $\omega$, $\omega^2$, $-1$, $1$ & $2$ & 26. \\
\hline
$c_{24}$ & $D_5(a_1)$ & $12$ & $-\omega$, $-\omega^2$, $\ii$, $-\ii$, $-1$, $1$ & $2$ & 30. \\
\hline
$c_{25}$ & $A_4 \times A_1$ & $10$ & $\xi_5$, $\xi_5^2$, $\xi_5^3$, $\xi_5^4$, $-1$, $1$ & $2$ & 24. \\
\hline

\end{tabular}
\end{center}

\caption[]{\label{table3} Conjugacy classes of elements in $W(E_6)$}
\end{table}

\begin{remark}
\label{DP2blowdown}
Theorem \ref{DP3} for cubic surfaces of types $(c_7)$, $(c_8)$, $(c_{15})$, $(c_{20})$, $(c_{22})$, $(c_{23})$, $(c_{24})$ and $(c_{25})$ immediately follows from Theorem \ref{DP2}. One can consider a del Pezzo surface of degree $2$ of the corresponding type (see the last column of Table \ref{table3}) and contract a line defined over $\F_q$.
\end{remark}

Many types of cubic surfaces are considered in \cite{Ry05}, \cite{SD10}, \cite{RT17}, \cite{Tr17}, \cite{BFL16} and in some lemmas in Section $4$. Actually, there are only three types of cubic surfaces left: $(c_2)$, $(c_9)$, $(c_{21})$. These cubic surfaces are isomorhic to the blowup of $\Pro^2_{\F_q}$ at a number of points of certain degrees. We consider these types in the following three lemmas.

\begin{lemma}
\label{TypeC2}
A cubic surface of type $(c_2)$ does not exist over $\F_2$, $\F_3$, and exists over other finite fields.
\end{lemma}

\begin{proof}

A cubic surface of type $(c_2)$ is the blowup of $\Pro^2_{\F_q}$ at two $\F_q$-points and two points of degree $2$.

In the proof of Lemma \ref{Type3} it is shown that there are $q^2 + 3q + 1$ points defined over $\F_q$ on $X$ and at least $6q + 5$ points defined over $\F_q$ lying on the lines on $X$. Therefore this surface does not exist for $q = 2$ and $q = 3$, since \mbox{$q^2 + 3q + 1 < 6q + 5$} in these cases.

For $q \geqslant 4$ the cubic surface $X$ can be constructed in the following way. Let $p_1$, $p_2$ and $p_3$, $p_4$ be two pairs of conjugate geometric points defined over $\F_{q^2}$ in general position on~$\Pro^2_{\F_q}$. Consider two smooth conics $Q_1$ and $Q_2$ defined over $\F_q$ passing through $p_1$, $p_2$, $p_3$ and $p_4$. Let $q_1$ be an $\F_q$-point on $Q_1$ and $q_2$ be an $\F_q$-point on $Q_2$. We show that the points $p_i$ and $q_j$ are in general position.

Any line passing through two points $p_i$ and $p_j$ does not contain other points lying on~$Q_1$ or $Q_2$. Therefore such line passes through exactly two points from the set $\{p_i, q_j\}$. If a point $p_i$ lie on a line $L$ passing through $q_1$ and $q_2$ then the conjugate of $p_i$ lie on this line since $q_1$ and $q_2$ are $\F_q$-points. Hence $L$ meets $Q_1$ at three points, that is impossible.

Thus the points $p_i$ and $q_j$ are in general position. The blowup of $\Pro^2_{\F_q}$ at these points is a cubic surface of type $(c_2)$ by Theorem \ref{GenPos} and Lemma \ref{blowup}.

\end{proof}

\begin{lemma}
\label{TypeC9}
A cubic surface of type $(c_9)$ does not exist over $\F_2$ and exists over other finite fields.
\end{lemma}

\begin{proof}

A cubic surface of type $(c_9)$ is the blowup of a $\Pro^2_{\F_q}$ at two points of degree $3$.

Let $q = 2$ and $P$ be a point on $\Pro^2_{\F_2}$ of degree $3$ in general position. There are $7$ smooth conics defined over $\F_2$ passing through $P$ and $7$ lines defined over $\F_2$ on $\Pro^2_{\F_2}$. These lines and conics contain $22$ points of degree three: $14$ points on lines, $7$ points that differs from~$P$ on conics passing through $P$, and $P$. But there are exactly $22$ points of degree $3$ on~$\Pro^2_{\F_2}$. Therefore there are no two points of degree $3$ in general position on $\Pro^2_{\F_2}$.

For $q \geqslant 3$ a cubic surface of type $(c_9)$ is constructed in the proof of \cite[Proposition~6.1]{RT17}.

\end{proof}

\begin{lemma}
\label{TypeC21}
A cubic surface of type $(c_{21})$ exists over all finite fields.
\end{lemma}

\begin{proof}

A cubic surface of type $(c_{21})$ is the blowup of $\Pro^2_{\F_q}$ at an $\F_q$-point, a point of degree~$2$ and a point of degree $3$.

Consider a smooth conic $Q$ on $\Pro^2_{\F_q}$. Let $q_1$ and $q_2$ be a pair of conjugate geometric points on $Q$ defined over $\F_{q^2}$, and $q_3$, $q_4$, $q_5$ be three conjugate geometric points on $Q$ defined over $\F_{q^3}$. Then any line $L_{ij}$ passing through two geometric points $q_i$ and $q_j$ except~$L_{12}$ does not contain $\F_q$-points.

There are $2q + 2$ points defined over $\F_q$ on $Q \cup L_{12}$, and $q^2 + q + 1$ points defined over~$\F_q$ on $\Pro^2_{\F_q}$. One has $q^2 + q + 1 > 2q + 2$ for $q \geqslant 2$. Therefore there is an $\F_q$-point $p$ on $\Pro^2_{\F_q}$ not lying on $L_{ij}$ and $Q$.

Thus the points $p$ and $q_i$ are in general position. The blowup of $\Pro^2_{\F_q}$ at these points is a cubic surface of type $(c_{21})$ by Theorem \ref{GenPos} and Lemma \ref{blowup}.

\end{proof}

For the three types $(c_{12})$, $(c_{13})$, and $(c_{14})$ of minimal cubic surfaces considered in \cite{RT17} there are some restrictions on $q$. More precisely, cubic surfaces of types $(c_{12})$ and $(c_{13})$ are constructed for any odd $q$, and cubic surfaces of type $(c_{14})$ are constructed for any $q$, such that $q = 6k + 1$. In the following lemma we construct cubic surfaces of types $(c_{12})$ and $(c_{13})$ for even $q$. Also for even $q$ this lemma give a construction of cubic surfaces of type $(c_{11})$, that are considered in \cite{SD10}.

\begin{lemma}
\label{TypesC12C13}
Cubic surfaces of types $(c_{12})$ and $(c_{13})$ exist over all finite fields.
\end{lemma}

\begin{proof}

For any odd $q$ cubic surfaces of types $(c_{12})$ and $(c_{13})$ are constructed in \cite[Subsection 5.3]{RT17} and \cite[Subsection 5.2]{RT17} respectively. We construct these types of cubics surfaces for any even $q$.

Let $q$ be even and $X$ be a cubic surface in $\Pro^3_{\F_q}$ given by the equation
$$
A(x^2z + xz^2 + y^3 + y^2z) + Byz(y+z) + Cz^3 + Dy(y+z)t + Ez^2t +Fzt^2 + t^3 = 0.
$$
\noindent One can check that the surface $X$ is smooth if and only if $A \ne 0$.

The automorphism $g$ of order $4$ given by
$$
g: (x : y : z : t) \mapsto (x + y : y + z : z : t)
$$
\noindent acts on $X$. The $g$-invariant hyperplane section $z = 0$ is the union of three lines
$$
Ay^3 + Dy^2t + t^3 = 0.
$$
\noindent Moreover, one can select $A$ and $D$ so that these three lines are conjugate. 

Therefore the group $\Gamma$ has order divisible by $3$ and commute with the image $\widetilde{g}$ of $g$ in $\Pic(\XX)$. In particular, an element $h'$ of order $3$ in $\Gamma$ commutes with $\widetilde{g}$. The element $\widetilde{g}h'$ has order $12$, thus the type of this element is either $c_{13}$, or $c_{24}$ (see Table \ref{table3}). Therefore the element $h'$ has type $c_6$ or $c_{11}$. By \cite[Lemma 2.10]{RT17} an element of type $c_6$ is conjugate to $(123) \in S_6 \subset W(E_6)$. Applying the notation of Remark \ref{cubicnotation} one can easily check that there are no $(123)$-orbits consisting of three meeting each other lines. Therefore $h'$ has type $c_{11}$ and $X$ is a minimal cubic surface of type $(c_{11})$, $(c_{12})$ or $(c_{13})$.

Let $h$ be a generator of $\Gamma$. Applying Proposition \ref{dPtwist} one can obtain cubic surfaces $X_2$, $X_3$ and $X_4$ such that the corresponding groups $\Gamma_2$, $\Gamma_3$ and $\Gamma_4$ in $W(E_6)$ are generated by $\widetilde{g}h$, $\widetilde{g}^2h$ and $\widetilde{g}^3h$ respectively. We want to show that there are a surface of type $(c_{11})$, a surface of type $(c_{12})$ and two surfaces of type $(c_{13})$ among the four surfaces $X$, $X_2$, $X_3$ and $X_4$.

If $X$ has type $(c_{11})$ or $(c_{12})$ then the surface $X_2$ has type $(c_{13})$, so for these two cases we can replace $X$ by $X_2$ and reduce these cases to the case when $X$ has type $(c_{13})$.

Assume that $X$ has type $(c_{13})$. The elements $h$ and $\widetilde{g}$ can be simultaneously diagonalized in $\mathrm{GL}\left(K_X^{\perp} \otimes \overline{\mathbb{Q}}\right)$, and $\widetilde{g}$ multiplies the eigenvalues of $h$ by $\ii$, $-\ii$, $\ii$, $-\ii$, $1$ and $1$. The element~$\widetilde{g}^2h$ has order $12$ and type $c_{13}$, therefore $\widetilde{g}$ trivially acts on the eigenvalues $\omega$ and $\omega^2$ of $h$. Thus all eigenvalues of the element $\widetilde{g}h$ are $\omega$ or $-\omega$, and this element has type $c_{11}$ or $c_{12}$. For these cases the element $\widetilde{g}^3h$ has type $c_{12}$ or $c_{11}$ respectively. Hence two of the surfaces $X$, $X_1$, $X_2$, and $X_3$ have type $(c_{13})$, one surface has type $(c_{12})$ and one surface has type $(c_{11})$.

\end{proof}


Now we prove Theorem \ref{DP3}.

\begin{proof}[Proof of Theorem \ref{DP3}]

For each type of cubic surfaces we give a reference, where this type is considered.

\begin{itemize}

\item The type $(c_1)$ is considered in \cite{SD10} (see also \cite[Corollary 3.3]{BFL16}).
\item The type $(c_2)$ is considered in Lemma \ref{TypeC2}.
\item The type $(c_3)$ is considered in the proof of Lemma \ref{Type10}.
\item The type $(c_4)$ is considered in the proof of \cite[Lemma 3.12]{Tr17}.
\item The type $(c_5)$ is considered in the proof of Lemma \ref{Type14}.
\item The type $(c_6)$ is considered in \cite[Subsection 3.1, case $a = 4$]{BFL16}.
\item The type $(c_9)$ is considered in Lemma \ref{TypeC9}.
\item The type $(c_{10})$ is considered in \cite[Propositions 6.1 and 6.5]{RT17}
\item The type $(c_{11})$ is considered in \cite{SD10}.
\item The type $(c_{12})$ is considered in Lemma \ref{TypesC12C13}.
\item The type $(c_{13})$ is considered in Lemma \ref{TypesC12C13}.
\item The type $(c_{14})$ is considered in \cite[Subsection 5.4]{RT17}.
\item The type $(c_{16})$ is considered in \cite{SD10} (see also \cite[Subsection 3.1, $a = 5$]{BFL16}).
\item The type $(c_{17})$ is considered in the proof of Lemma \ref{Type6}.
\item The type $(c_{18})$ is considered in \cite[Subsection 3.1, case $a = 3$]{BFL16}.
\item The type $(c_{19})$ is considered in \cite[Theorem 3.2]{Ry05}.
\item The type $(c_{21})$ is considered in Lemma \ref{TypeC21}.

\end{itemize}

The types $(c_7)$, $(c_8)$, $(c_{15})$, $(c_{20})$, $(c_{22})$, $(c_{23})$, $(c_{24})$ and $(c_{25})$ are considered in Remark~\ref{DP2blowdown}.

\end{proof}

\section{Del Pezzo surfaces of degree $4$}

In this section for each cyclic subgroup $\Gamma \subset W(D_5)$ we construct the corresponding del Pezzo surface of degree $4$ over $\F_q$ if it is possible, and show that such surfaces do not exist for other values of $q$. As a result we prove Theorem \ref{DP4}.

In Table \ref{table4} we collect some facts about conjugacy classes of elements in the Weyl group~$W(D_5)$. The Weyl group $W(D_5)$ is isomorphic to $\left(\Z /2\Z\right)^4 \rtimes S_5 \subset \mathrm{GL}_5(\Z)$. The elements of $\left(\Z /2\Z\right)^4$ change signes of even number of the coordinates, and the elements of~$S_5$ permute the coordinates (see \cite[Subsection 6.4]{DI09}). One can find the classification of conjugacy classes in $\left(\Z /2\Z\right)^4 \rtimes S_5$ in \cite[Table 3]{DI09}, but there are two missed classes (denoted by $(ab)\iota_{ac}$ and $(ab)\iota_{acde}$ in Table \ref{table4}). We give a corrected classification in Table~\ref{table4} and prove that there are no missed classed in Proposition \ref{nomistakeDP4}.

The first column is a conjugacy class in the notation of \cite[Subsection 6.4]{DI09}. The second column is the Carter graph corresponding to the conjugacy class (see \cite{Car72}). The third column is the order of element. The fourth column is the collection of eigenvalues of the action of element on $K_X^{\perp} \subset \Pic(\XX) \otimes \mathbb{Q}$. The fifth column is the invariant Picard number $\rho(\XX)^{\Gamma}$. The last column is the type of the corresponding conjugacy class (see Table \ref{table3}) of $W(E_6)$ after blowing up a del Pezzo surface of degree $4$ at an $\F_q$-point.

\begin{table}

\begin{center}
\begin{tabular}{|c|c|c|c|c|c|}
\hline
Class & Graph & Order & Eigenvalues & $\rho(\XX)^{\Gamma}$ & Blowup \\
\hline
$\mathrm{id}$ & $\varnothing$ & $1$ & $1$, $1$, $1$, $1$, $1$ & $6$ & $c_{1}$ \\
\hline
$(ab)$ & $A_1$ & $2$ & $-1$, $1$, $1$, $1$, $1$ & $5$ & $c_{16}$ \\
\hline
$(ab)(cd)$ & $A_1^2$ & $2$ & $-1$, $-1$, $1$, $1$, $1$ & $4$ & $c_{2}$ \\
\hline
$\iota_{ab}$ & $A_1^2$ & $2$ & $-1$, $-1$, $1$, $1$, $1$ & $4$ & $c_{2}$ \\
\hline
$(abc)$ & $A_2$ & $3$ & $\omega$, $\omega^2$, $1$, $1$, $1$ & $4$ & $c_{6}$ \\
\hline
$(ab)\iota_{cd}$ & $A_1^3$ & $2$ & $-1$, $-1$, $-1$, $1$, $1$ & $3$ & $c_{17}$ \\
\hline
$(abcd)$ & $A_3$ & $4$ & $\ii$, $-\ii$, $-1$, $1$, $1$ & $3$ & $c_{18}$ \\
\hline
$(ab)\iota_{ac}$ & $A_3$ & $4$ & $\ii$, $-\ii$, $-1$, $1$, $1$ & $3$ & $c_{18}$ \\
\hline
$(abc)(de)$ & $A_2 \times A_1$ & $6$ & $\omega$, $\omega^2$, $-1$, $1$, $1$ & $3$ & $c_{21}$ \\
\hline
$\iota_{abcd}$ & $A_1^4$ & $2$ & $-1$, $-1$, $-1$, $-1$, $1$ & $2$ & $c_{3}$ \\
\hline
$(ab)(cd)\iota_{ae}$ & $A_3 \times A_1$ & $4$ & $\ii$, $-\ii$, $-1$, $-1$, $1$ & $2$ & $c_{5}$ \\
\hline
$(ab)(cd)\iota_{ac}$ & $D_4(a_1)$ & $4$ & $\ii$, $\ii$, $-\ii$, $-\ii$, $1$ & $2$ & $c_{4}$ \\
\hline
$(abcde)$ & $A_4$ & $5$ & $\xi_5$, $\xi_5^2$, $\xi_5^3$, $\xi_5^4$, $1$ & $2$ & $c_{15}$ \\
\hline
$(abc)\iota_{de}$ & $A_2 \times A_1^2$ & $6$ & $\omega$, $\omega^2$, $-1$, $-1$, $1$ & $2$ & $c_{8}$ \\
\hline
$(abc)\iota_{ad}$ & $D_4$ & $6$ & $-\omega$, $-\omega^2$, $-1$, $-1$, $1$ & $2$ & $c_{7}$ \\
\hline
$(ab)\iota_{acde}$ & $A_3 \times A_1^2$ & $4$ & $\ii$, $-\ii$, $-1$, $-1$, $-1$ & $1$ & $c_{19}$ \\
\hline
$(abcd)\iota_{ae}$ & $D_5$ & $8$ & $\xi_8$, $\xi_8^3$, $\xi_8^5$, $\xi_8^7$, $-1$ & $1$ & $c_{20}$ \\
\hline
$(abc)(de)\iota_{ad}$ & $D_5(a_1)$ & $12$ & $-\omega$, $-\omega^2$, $\ii$, $-\ii$, $-1$ & $1$ & $c_{24}$ \\
\hline

\end{tabular}
\end{center}

\caption[]{\label{table4} Conjugacy classes of elements in $W(D_5)$}
\end{table}

\begin{proposition}
\label{nomistakeDP4}
All conjugacy classes of $\left(\Z /2\Z\right)^4 \rtimes S_5$ are listed in Table \ref{table4}.
\end{proposition}

\begin{proof}
Let $g$ be an element in $\left(\Z /2\Z\right)^4 \rtimes S_5$. The image of $g$ under the natural homomorphism $f: \left(\Z /2\Z\right)^4 \rtimes S_5 \rightarrow S_5$ belongs to one of the seven conjugacy classes in $S_5$ parametrised by Young diagrams: trivial, $(12)$, $(123)$, $(12)(34)$, $(1234)$, $(12345)$, $(123)(45)$.

If $f(g)$ is trivial then $g$ is trivial or $S_5$-conjugate to $\iota_{12}$ or $\iota_{1234}$. These three cases are listed in Table \ref{table4}.

If $f(g)$ is conjugate to $(12)$ then $g$ is $S_5$-conjugate to $(12)$, $(12)\iota_{12}$, $(12)\iota_{13}$, $(12)\iota_{34}$, $(12)\iota_{1234}$ or $(12)\iota_{1345}$ One has $\iota_{13}(12)\iota_{12}\iota_{13} = (12)$ and $\iota_{13}(12)\iota_{1234}\iota_{13} = (12)\iota_{34}$. The other four cases are listed in Table \ref{table4}.

If $f(g)$ is conjugate to $(123)$ then $g$ is $S_5$-conjugate to $(123)$, $(123)\iota_{12}$, $(123)\iota_{14}$, $(123)\iota_{45}$, $(123)\iota_{1234}$ or $(123)\iota_{1245}$. One has $\iota_{24}(123)\iota_{12}\iota_{24} = (123)$, $\iota_{34}(123)\iota_{1234}\iota_{34} = (123)\iota_{14}$ and~$\iota_{24}(123)\iota_{1245}\iota_{24} = (123)\iota_{45}$. The other three cases are listed in Table \ref{table4}.

If $f(g)$ is conjugate to $(12)(34)$ then $g$ is $S_5$-conjugate to $(12)(34)$, $(12)(34)\iota_{12}$, $(12)(34)\iota_{13}$, $(12)(34)\iota_{15}$, $(12)(34)\iota_{1234}$ or $(12)(34)\iota_{1345}$ One has $\iota_{15}(12)(34)\iota_{12}\iota_{15} = (12)(34)$, $\iota_{13}(12)(34)\iota_{1234}\iota_{13} = (12)(34)$ and $\iota_{35}(12)(34)\iota_{1345}\iota_{35} = (12)(34)\iota_{15}$. The other three cases are listed in Table \ref{table4}.

If $f(g)$ is conjugate to $(1234)$ then $g$ is $S_5$-conjugate to $(1234)$, $(1234)\iota_{12}$, $(1234)\iota_{13}$, $(1234)\iota_{15}$, $(1234)\iota_{1234}$ or $(1234)\iota_{1235}$. One has $\iota_{25}(1234)\iota_{12}\iota_{25} = (1234)$, \mbox{$\iota_{23}(1234)\iota_{13}\iota_{23} = (1234)$}, $\iota_{13}(1234)\iota_{1234}\iota_{13} = (1234)$ and $\iota_{35}(1234)\iota_{1235}\iota_{35} = (1234)\iota_{15}$. The other two cases are listed in Table \ref{table4}.

If $f(g)$ is conjugate to $(12345)$ then $g$ is $S_5$-conjugate to $(12345)$, $(12345)\iota_{12}$, $(12345)\iota_{13}$ or $(12345)\iota_{1234}$. One has $\iota_{1345}(12345)\iota_{12}\iota_{1345} = (12345)$, \mbox{$\iota_{23}(12345)\iota_{13}\iota_{23} = (12345)$} and~$\iota_{24}(12345)\iota_{1234}\iota_{24} = (12345)$. The remaining case is listed in Table \ref{table4}.

If $f(g)$ is conjugate to $(123)(45)$ then $g$ is $S_5$-conjugate to $(123)(45)$, $(123)(45)\iota_{12}$, $(123)(45)\iota_{14}$, $(123)(45)\iota_{45}$, $(123)(45)\iota_{1235}$ or $(123)(45)\iota_{1245}$. One has $\iota_{1345}(123)(45)\iota_{12}\iota_{1345} = (123)(45)$, $\iota_{1234}(123)(45)\iota_{45}\iota_{1234} = (123)\iota_{14}$, \mbox{$\iota_{34}(123)(45)\iota_{1235}\iota_{34} = (123)(45)\iota_{14}$} and $\iota_{24}(123)(45)\iota_{1245}\iota_{24} = (123)(45)$. The other two cases are listed in Table~\ref{table4}.

\end{proof}

\begin{remark}
\label{DP3blowdown}
The result of Theorem \ref{DP4} for del Pezzo surfaces of degree $4$ of types, that differ from $\mathrm{id}$, $(ab)(cd)$, $\iota_{ab}$, $\iota_{abcd}$, $(ab)(cd)\iota_{ae}$, $(ab)(cd)\iota_{ac}$ , immediately follows from Theorem \ref{DP3}. One can consider a del Pezzo surface of degree $3$ of the corresponding type (see the last column of Table \ref{table4}) and contract a line defined over $\F_q$.
\end{remark}

The type $\mathrm{id}$ is considered in \cite[Lemma 3.1]{BFL16}, and the types $\iota_{abcd}$ and $(ab)(cd)\iota_{ac}$ of del Pezzo surfaces of degree $4$ are considered in \cite{Ry05} and \cite{Tr17}. The type $(ab)(cd)$ is isomorphic to the blowup of $\Pro^2_{\F_q}$ at an $\F_q$-point and two points of degree $2$. We consider this type in the following lemma.

\begin{lemma}
\label{Type(ab)(cd)}
A del Pezzo surface of degree $4$ of type $(ab)(cd)$ does not exist over $\F_2$ and exists over other finite fields.
\end{lemma}

\begin{proof}

A del Pezzo surface of degree $4$ of type $(ab)(cd)$ is the blowup of a $\Pro^2_{\F_q}$ at an \mbox{$\F_q$-point} and two points of degree $2$. Five geometric points on $\Pro^2_{\F_q}$ are in general position if and only if they lie on a smooth conic by \cite[Lemma 2.4]{BFL16}.

There are no an $\F_q$-point and two points of degree $2$ on a smooth conic $Q$ in $\Pro^2_{\F_q}$ for~$q = 2$. For other values of $q$ one can blow up points of these degrees on $Q$ and get a del Pezzo surface of degree $4$ of type $(ab)(cd)$ by Theorem \ref{GenPos} and Lemma \ref{blowup}.
\end{proof}

The two remaining cases are the blowups of a quadric surface in $\Pro^3_{\F_q}$ at four geometric points. This surface is isomorphic to $\Pro^1_{\overline{\F}_q} \times \Pro^1_{\overline{\F}_q}$ over $\overline{\F}_q$. We denote by $\pi_1$ and $\pi_2$ the projections on the first and the second factors of $\Pro^1_{\overline{\F}_q} \times \Pro^1_{\overline{\F}_q}$. The Picard group $\Pic\left(\Pro^1_{\overline{\F}_q} \times \Pro^1_{\overline{\F}_q}\right)$ is generated by the classes $F_1$ and $F_2$ of fibres of $\pi_1$ and $\pi_2$ respectively. Thus any divisor~$D$ on $\Pro^1_{\overline{\F}_q} \times \Pro^1_{\overline{\F}_q}$ is linearly equivalent to $aF_1 + bF_2$. The pair $(a, b)$ is called \textit{bedegree} of $D$. For example the anticanonical class $-K_{\Pro^1_{\overline{\F}_q} \times \Pro^1_{\overline{\F}_q}}$ has bedegree $(2, 2)$.

The next well-known proposition defines points in general position on $\Pro^1_{\overline{\F}_q} \times \Pro^1_{\overline{\F}_q}$.

\begin{proposition}
\label{DP8GenPos}
Let $1 \leqslant d \leqslant 8$, and $p_1$, $\ldots$, $p_{8-d}$ be $8-d$ points on $\Pro^1_{\kka} \times \Pro^1_{\kka}$ such that
\begin{itemize}
\item no two lie on a fibre of $\pi_1$ or $\pi_2$;
\item no four lie on a curve of bedegree $(1, 1)$;
\item no six lie on a curve of bedegree $(1, 2)$ or $(2, 1)$;
\item for $d = 1$ the points are not on a singular curve of bedegree $(2, 2)$ with singularity at one of these points.
\end{itemize}
Then the blowup of $\Pro^1_{\kka} \times \Pro^1_{\kka}$ at $p_1$, $\ldots$, $p_{8-d}$ is a del Pezzo surface of degree $d$.

Moreover, any del Pezzo surface $\XX$ of degree $1 \leqslant d \leqslant 7$ over algebraically closed field $\kka$ is the blowup of such set of points.
\end{proposition}

\begin{definition}
\label{DP8GenPosdef}
As in Definition \ref{GenPosdef} we say that a collection of points on a smooth quadric $Q \subset \Pro^3_{\F_{\ka}}$ is \textit{in general position} if it satisfies the conditions of Proposition \ref{DP8GenPos}.
\end{definition}

Now we can consider the cases $\iota_{ab}$ and $(ab)(cd)\iota_{ae}$ of Table \ref{table4}.

\begin{lemma}
\label{TypeIab}
A del Pezzo surface of degree $4$ of type $\iota_{ab}$ does not exist over $\F_2$, $\F_3$, and exists over other finite fields.
\end{lemma}

\begin{proof}

A del Pezzo surface of degree $4$ of type $\iota_{ab}$ is the blowup of $\Pro^1_{\F_q} \times \Pro^1_{\F_q}$ at two points of degree $2$.

For $q = 2$ there are only two geometric fibres of $\pi_1$ defined over $\F_4$ and not defined over~$\F_2$. Therefore one can not find two points of degree $2$ in general position for this case.

For $q = 3$ there are $18$ points of degree $2$ not lying in one geometric fibre of $\pi_1$ or $\pi_2$ on~$\Pro^1_{\F_3} \times \Pro^1_{\F_3}$, that are the points of degree $2$ in general position. Let $P$ be one of these points. Then the union of the two geometric fibres of $\pi_1$ and the two geomertic fibres of $\pi_2$ passing through the two geometric points corresponding to $P$ contains $10$ points of degree~$2$ in general position (including $P$). There are four curves of bedegree $(1, 1)$ defined over $\F_3$ and passing through $P$. These curves contain $8$ points of degree $2$ in general position that differ from $P$. Therefore any point of degree $2$ in general position, either lies on one of these curves, or lies in the geometric fibres of $\pi_1$ or $\pi_2$ containing $P$. Thus there are no pairs of points of degree $2$ in general position on $\Pro^1_{\F_3} \times \Pro^1_{\F_3}$.

For $q \geqslant 4$ one can consider a cubic surface of type $(c_{2})$ (this surface is constructed in Lemma \ref{TypeC2}), contract a line defined over $\F_q$ and get a del Pezzo surface of degree $4$ of type $\iota_{ab}$.

\end{proof}

\begin{lemma}
\label{Type(ab)(cd)Iae}
A del Pezzo surface of degree $4$ of type $(ab)(cd)\iota_{ae}$ does not exist over $\F_2$ and exists over other finite fields.
\end{lemma}

\begin{proof}

A del Pezzo surface of degree $4$ of type $(ab)(cd)\iota_{ae}$ is the blowup of a conic $Q \subset \Pro^3_{\F_q}$, such that $\rho(Q) = 1$ at a point of degree $4$.

For $q = 2$ there are five $\F_2$-points and $66$ points of degree $4$ on $Q$. There are five pairs defined over $\F_4$ conjugate fibres of $\pi_1$ and $\pi_2$ each containing $6$ points of degree $4$, ten curves of bedegree $(1, 1)$ passing through three $\F_2$-points each containing $3$ points of degree $4$, and six points of degree $4$, that are the intersection of two geometric fibres $F_{11}$ and $F_{12}$ of $\pi_1$ and two geometric fibres $F_{21}$ and $F_{22}$ of $\pi_2$ such that the fibres $F_{11}$, $F_{12}$, $F_{21}$ and $F_{22}$ are permuted by the Galois group $\Gal\left(\F_{16}/\F_2\right)$. One has $5 \cdot 6 + 10 \cdot 3 + 6 = 66$. Thus there is no a point of degree $4$ in general position on $Q$ for $q = 2$.

For $q \geqslant 3$ one can consider a cubic surface of type $(c_{5})$ (this surface is constructed in the proof of Lemma \ref{Type14}), contract a line defined over $\F_q$ and get a del Pezzo surface of degree $4$ of type $(ab)(cd)\iota_{ae}$.

\end{proof}

Now we prove Theorem \ref{DP4}.

\begin{proof}[Proof of Theorem \ref{DP4}]

For each type of del Pezzo surfaces of degree $4$ we give a reference, where this type is considered.

\begin{itemize}

\item The type $\mathrm{id}$ is considered in \cite[Lemma 3.1]{BFL16}.
\item The type $(ab)(cd)$ is considered in Lemma \ref{Type(ab)(cd)}.
\item The type $\iota_{ab}$ is considered in Lemma \ref{TypeIab}.
\item The type $\iota_{abcd}$ is considered in \cite[Theorem 3.2]{Ry05} and \cite[Theorem 2.5]{Tr17}.
\item The type $(ab)(cd)\iota_{ae}$ is considered in Lemma \ref{Type(ab)(cd)Iae}.
\item The type $(ab)(cd)\iota_{ac}$ is considered in \cite[Theorem 3.2]{Ry05}.

\end{itemize}

The other types are considered in Remark~\ref{DP3blowdown}.

\end{proof}

\section{Open questions}

In this section we discuss open questions that arise for constructing del Pezzo surfaces of degree $1$ over finite fields.

By \cite[Table 11]{Car72} there are $112$ types of del Pezzo surfaces of degree $1$. This number is greater than the sum of the numbers of del Pezzo surfaces types of degrees $2$, $3$ and $4$. For $30$ types of del Pezzo surfaces $X$ of degree $1$ one has $\rho(X) = 1$.

As in the case of del Pezzo surfaces of degree $2$ one can define the \textit{Bertini twist}, since each del Pezzo surface~$X$ of degree $1$ has an involution defined by the double cover of~$\Pro_{\F_q}(1 : 1 : 2)$ given by the linear system $|-2K_X|$. But the Bertini twist $X_2$ of a del Pezzo surface $X_1$ of degree $1$ such that $\rho(X_1) = 1$ can have $\rho(X_2) = 1$. Actually, for five pairs of types of del Pezzo surfaces $X_1$ and $X_2$ of degree $1$ one has $\rho(X_1) = \rho(X_2) = 1$. Moreover, there are seven types of del Pezzo surfaces $X$ of degree $1$ with $\rho(X) = 1$, such that the Bertini twist of~$X$ has the same type as $X$.

For nonminimal del Pezzo surfaces of degree $1$ the methods that works for del Pezzo surfaces of degree at least $2$ does not work. To apply Theorem \ref{GenPos} one has to check that eight geometric points of the blowup do not lie on a cubic curve having a singularity at one of these points. Also to apply Corollary \ref{dPblowup} to del Pezzo surfaces $X$ of degree $2$ one has to consider an additional condition, that the point of the blowup do not lie on the ramification divisor of the anticanonical map $X \rightarrow \Pro^2_{\F_q}$. For even $q$ there arise additional difficulties (see \cite[Lemma 4.1]{BFL16}). Moreover, on del Pezzo surfaces of degree $2$ a pair of lines permutted by the Geiser involution can have two or one common geometric points, and there can be four lines meeting each other either in six distinct geometric points, or in one common geometric point that is called a \textit{generalised Eckardt point}. Therefore for del Pezzo surfaces of degree $2$ formulas for calculating $\F_q$-points lying on the lines are much more complicated than the formulas given in Remark \ref{Linespoints}.

To construct some types of del Pezzo surfaces of degree $1$ that admit a structure of a minimal conic bundle one can apply \cite[Theorem 2.11]{Ry05}. But a constructed minimal surface admitting a structure of a conic bundle is not in general a del Pezzo surface, as in the cases of degrees $4$ and $2$ (see the proof of \cite[Theorem 3.2]{Ry05} and \cite[Proposition~2.3]{Tr17} respectively). 


Despite these problems some types of del Pezzo surfaces of degree $1$ are constructed in \cite[Section 5]{BFL16} and \cite[Lemma 3.11]{Tr17}.

A generalization of the results of this paper to the case of del Pezzo surfaces of degree~$1$ is a great challenge.

\bibliographystyle{alpha}

\end{document}